%% file: robustNIPS15.tex
\documentclass{article} 
\usepackage{hyperref}
\let\chapter\section
 \usepackage[ruled,linesnumbered,boxed]{algorithm2e}
\usepackage{graphicx}
\usepackage{amssymb,amsmath,amscd,mathrsfs}
\usepackage{amsthm}
\usepackage{comment,color}
\usepackage{subfigure}
\usepackage{multirow}
\usepackage{etoolbox}
\usepackage{macros}
\usepackage{zref-xr}
\usepackage[capitalize]{cleveref}
\usepackage{tikz}
\usetikzlibrary{shapes.arrows}
\usetikzlibrary{arrows,automata,backgrounds,positioning}

\title{Robust Policy Optimization with Baseline Guarantees}
\author{
Yinlam Chow\thanks{Institute of Computational \& Mathematical Engineering, Stanford University. Stanford CA, USA 94305},\,\, Marek Petrik\thanks{IBM T. J. Watson Research Center, Yorktown NY, USA 10598},\,\, Mohammad Ghavamzadeh\thanks{Adobe Research, San Jose CA, USA 95110} }

\newcommand{\opt}{^{\star}}
\newcommand{\polbase}{\pi_{\text{B}}}	

\newcommand{\eye}{\mathbf{I}}

\newcommand{\tr}{^{\mathsf{T}}}
\newcommand{\pol}{\pi}
\newcommand{\indist}{p_0}
\newcommand{\return}{\rho}
\newcommand{\disc}{\gamma}
\newcommand{\rmax}{R_{\max}}

\input{defs}

\input{RL_var}

\begin{document}

\maketitle


\vspace{-0.25in}
\begin{abstract}
Our goal is to compute a policy that guarantees improved return over a baseline policy even when the available MDP model is inaccurate. The inaccurate model may be constructed, for example, by system identification techniques when the true model is inaccessible. When the modeling error is large, the standard solution to the constructed model has no performance guarantees with respect to the true model.  In this paper we develop algorithms that provide such performance guarantees and show a trade-off between their \emph{complexity} and \emph{conservatism}. Our novel model-based safe policy search algorithms leverage recent advances in robust optimization techniques. Furthermore we illustrate the effectiveness of these algorithms using a numerical example.
\end{abstract}

\vspace{-0.1in}
\section{Introduction}\label{sec:intro}
\vspace{-0.1in}
Many problems in science and engineering can be formulated as sequential decision-making under uncertainty. A common scenario in such problems in many different areas, such as online marketing, inventory control, health informatics, and computational finance, is that we are given a batch of data generated by the current strategy(ies) of the company (hospital, investor), and we are asked to find a good or an optimal policy. Although there are many techniques to find a good policy given a batch of data, there are not much results to guarantee that the obtained policy will perform well in the real system without deploying it. Since deploying an untested policy might be risky and harmful for the business, the product (hospital, investment) manager does not allow it unless we provide her with some sort of guarantees on the performance of the policy, e.g.,~convince her that the policy performs at least as well her existing strategy. 

Our focus is on a {\em model-based} approach to this fundamental problem. In this approach, we first use the batch of data and build a {\em simulator} that mimics the behavior of the dynamical system under studies (online advertisement, inventory system, emergency room of a hospital, financial market), together with an {\em error function} that bounds its accuracy, and then use this simulator to generate data and learn a (good) policy. The main challenge here is to have guarantees on the performance of the learned policy, given the error in the simulator. This line of research is closely related to the area of {\em robust} learning and control. What makes our problem different than the standard robust learning and control~\cite{ahmed2013regret,even2003action,hanasusanto2014two,iyengar2005robust,kalyanasundaram2002markov,li2007robust,le2007robust,wiesemann2013robust} is the existence of a {\em baseline} policy (e.g.,~current company's strategy), which often has a reasonable performance. This difference allows us to develop algorithms, whose performance (the performance of their returned policy) is better than the standard robust methods that optimize for the worst-case scenario.    

In this work, we assume that {1)} the sequential decision-making problem can be modeled as an \emph{infinite-horizon} Markov decision process (MDP); {2)} we are given a {\em simulator} of this system together with an {\em error function} that bounds its accuracy (we briefly discuss how the simulator and error function can be built from the batch of data in Appendix~\ref{subsec:sampling_bounds}); {3)} we are given a {\em baseline policy} for the problem (e.g.,~the current strategy of the company); and {4)} the performance of the baseline policy ({\em baseline performance}) is known (this is a reasonable assumption as the batch of data is often large enough to have an accurate estimate of the performance of its generating strategy), and our goal is to find a policy that is {\em safe}, i.e.,~performs at least as well as the base policy in the real-world. We present four algorithms to tackle this problem in Sections~\ref{subsec:reward-adjusted} to~\ref{sec:combined}. For each algorithm, we prove that its returned policy is safe and provide a bound on its performance loss w.r.t.~an optimal policy of the real system. From each proposed algorithm to the next, the \emph{computational complexity} grows, but at the same time, the chance of finding a safe policy other than the obvious solution of the baseline policy also increases (the safe policy search become less \emph{conservative}). These major findings are summarized in Figure \ref{fig:flow}, whose the notations are clearly defined in latter sections. We show this change in the behavior of our algorithms through a simple example in Section \ref{sec:experiments}. This example also serves as a proof of concept for the algorithms in Section \ref{subsec:reward-adjusted} to \ref{sec:combined}. Another important difference between the algorithms in Section \ref{subsec:reward-adjusted} to~\ref{subsec:aug_robust} and the one in Section \ref{sec:combined} is that the latter directly works with the baseline policy, while the former uses it in an indirect way, and works with the baseline performance. 
\begin{figure*}[h!tpb]\footnotesize{
\begin{tikzpicture}[thick,scale=0.8, every node/.style={scale=0.8}]
  \node[draw, align = center] (begin) at (0,0.5) {Find a safe policy $\pi$, i.e.,\\ $\rho(\pi,\M^\star)\geq\rho(\pi_B,\M^\star)$?};
   \node[align=center] (q1) at (13.5,2) {{\small Low $C1$ High $C2$}};
 \node[single arrow,draw=black,fill=black!10,minimum height=2cm,shape border rotate=270] at (13.5,0.5) {};
  \node[align=center] (q1) at (13.5,-1) {{\small High $C1$ Low C2}};
  \node[draw,align=center,fill=green] (RA) at (6.5,2) {RaMDP\\Section \ref{subsec:reward-adjusted}};
   \node[draw,align=center,fill=yellow] (R) at (6.5,1) {RMDP\\Section \ref{subsec:robust}};
  \node[draw,align=center,fill=cyan] (AR) at (6.5,-0) {Augmented RMDP\\Section \ref{subsec:aug_robust}};
   \node[draw,align=center,fill=magenta] (I) at (6.5,-1) {Combine Robust and Baseline Policies\\Section \ref{sec:combined}};
    
  \draw[->,draw=black] (begin) -- (RA) node [midway, above, sloped] {{\small Standard MDP}};
  \draw[->,draw=black] (begin) -- (R) node [midway, above, sloped] {{\small RMDP}};
  \draw[->,draw=black] (begin) -- (AR) node [midway, above, sloped] {{\small Iterative RMDP}};
  \draw[->,draw=black] (begin) -- (I) node [midway, above, sloped] {{\small DRO}};
  
   \node[draw,align=center,fill=green] (piRA) at (11.5,2) {Policy \\$\pi_{Ra}$};
   \node[draw,align=center,fill=yellow] (piR) at (11.5,1) {Policy \\$\pi_{R}$};
  \node[draw,align=center,fill=cyan] (piAR) at (11.5,-0) {Policy \\$\pi_{AR}$};
   \node[draw,align=center,fill=magenta] (piI) at (11.5,-1) {Policy \\$\pi_{I}$};

  \draw[->,draw=black] (RA) -- (piRA) node [midway, above, sloped] {$A/R$};
  \draw[->,draw=black] (R) -- (piR) node [midway, above, sloped] {$A/R$};
  \draw[->,draw=black] (AR) -- (piAR) node [midway, above, sloped] {$A/R$};
  \draw[->,draw=black] (I) -- (piI) node [midway, above, sloped] {Solution};
  
  \node[align=center] (q1) at (0,-0.5) {{\small ($C1$): Computation Complexity}};
  \node[align=center] (q2) at (0,-1) {{\small ($C2$): Conservatism}};
\end{tikzpicture}
\caption{An overview of various methods for finding \emph{safe} policies. Here $A/R$ stands for acceptance/rejection of the offline test. The arrow illustrates the computational complexity/conservatism trade-off of different methods.}}
\label{fig:flow}
\end{figure*}
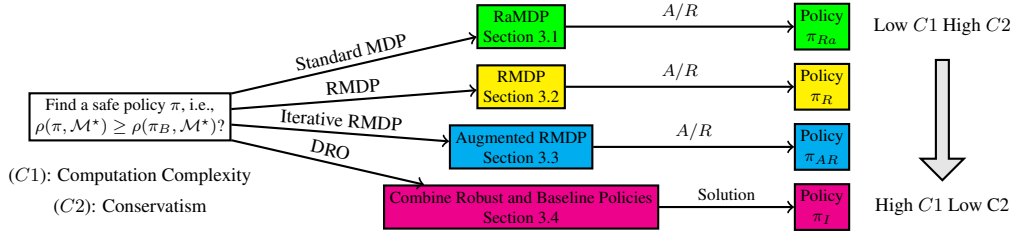

\section{Preliminaries}  \label{sec:prelimMDP}
\vspace{-0.1in}
A $\gamma$-discounted MDP is a tuple $\M=\langle\X,\A,r,P,p_0,\gamma\rangle$, where $\X$ and $\A$ are the state and action spaces; $r(x,a)\in[-R_{\max},R_{\max}]$ is the bounded reward function; $P(\cdot|x,a)$ is the transition probability function; and $p_0(\cdot)$ is the initial state distribution. A solution to MDP $\M$ is a stationary policy $\pi$ that is a mapping from $\M$'s state space to a distribution over its action space, i.e.,~$\pi:\X\times\A\rightarrow [0,1]$. We denote by $\Pi_S$, the set of all such policies for MDP $\M$. We define the performance of $\pi$ in the world modeled by MDP $\M$ as 
\begin{equation*}
\rho(\pi,\M)=\lim_{T\rightarrow \infty} \mathbb E\left(\sum_{t=0}^{T-1}\gamma^t r\big(X_t,\pi(X_t)\big)\mid p_0\right) = p_0^\top V^\pi_{\M},
\end{equation*}
where $X_t$ is the random variable representing the state of the MDP $\M$ at time-step $t$ and $V^\pi_{\M}$ is the value function of policy $\pi$ in $\M$. We also define an optimal policy as $\pi^*\in\argmax_{\pi\in\Pi_S}\rho(\pi,\M)$.

As mentioned in Section~\ref{sec:intro}, we use the historical data to build a {\em simulator} of the system, together with an error function that measures its accuracy. We denote by $\M^\star$ and $\widehat{\M}$ the (unknown) {\em true} and {\em simulated} MDPs with transition probability functions $P^\star$ and $\widehat{P}$, respectively.\footnote{In this paper, we restrict our attention to error in the transition probability function to simplify the exposition; the results readily extend to the case with error in the reward function.} In order to capture the deviation between these models, we make the following assumption:
\begin{assumption} 
\label{asm:error}
For each $(x,a)\in\mathcal{X}\times \mathcal{A}$, the error function $e(x,a)$ bounds the $L_1$-norm of the difference between the true and estimated transition probabilities, i.e.,
\[ \|P^\star(\cdot|x,a)-\widehat{P}(\cdot|x,a)\|_1 \le e(x,a)~ . \]
\end{assumption}
In many practical situations, the deviation between $P^\star$ and $\widehat P$ is bounded with high probability. Nevertheless, here we restrict the error bound to be deterministic to simplify the analysis in latter sections. Extending these results with probabilistic bounds is direct and omitted for brevity. More information on building the simulator and computing the $L_1-$deviation bound (using an empirical distribution for simulator $\widehat P$ and the Weissman distribution bound~\cite{weissman2003inequalities}) is available in Appendix~\ref{subsec:sampling_bounds}.
%
\begin{remark}
Using the estimated transition probability function $\widehat{P}$ and the error function $e$, we may construct the {\em uncertainty set}
\[ \U(\widehat{P},e) = \left\{P\;:\;\|P(\cdot|x,a)-\widehat{P}(\cdot|x,a)\|_1\leq e(x,a),\;\forall x,a\in\X\times\A\right\}~. \]
This uncertainty set automatically defines an uncertainty set $\U(\widehat{\M},e)$ for MDPs. It is clear that the true transition probability function $P^\star$ (the true MDP $\M^\star$) belongs to the uncertainty set $\U(\widehat{P},e)$.
\end{remark}
Given Assumption \ref{asm:error}, we have the following upper-bound on the difference between the return of a policy $\pi$ in the true and simulated MDPs (with transition probabilities $P^\star$ and $\widehat{P}$, respectively).
\begin{lemma} \label{cor:policy_tran_error}
Given Assumption \ref{asm:error}, for any policy $\pi$, the difference between the return of a policy $\pi$ in the true and simulated MDPs is upper bounded as follows:
\[  |\return(\pi,\widehat{\M}) - \return(\pi,\M^\star)| \le \frac{\gamma\rmax}{1-\gamma} p_0\tr (\eye - \disc P^\star_\pi)^{-1} e_\pi,\]
where $P^\star_\pi$ and $e_\pi$ are the transition probability of the true MDP and the error function when the actions are taken according to policy $\pi$. \hfill{\bf{\em(Appendix~\ref{appendix:LemmaPolTranErr})}}
\end{lemma} 

As discussed in Section~\ref{sec:intro}, we assume that we are provided with a {\em baseline policy} $\pi_B$ and we have a very good approximation of its performance $\rho(\pi_B,\M^\star)$. We call a policy $\pi$ {\em safe} if it is guaranteed to perform not worse than the baseline policy in the true MDP $\M^\star$, i.e.,~$\rho(\pi,\M^\star)\ge\rho(\pi_B,\M^\star)$. In the next section, we aim to explore several methods to find a safe policy $\pi$, given the simulator $\widehat{\M}$, the error function $e$, the baseline policy $\pi_B$, and the baseline performance $\rho(\pi_B,\M^\star)$. For each method, we provide a bound on the performance loss of its returned policy $\pi$ w.r.t.~an optimal policy of the true MDP $\pi^*_{\M^\star}$, i.e.,~$\Phi(\pi)\stackrel{\Delta}{=}\rho(\pi^*_{\M^\star},\M^\star)-\rho(\pi,\M^\star)$. 

\section{Computing Safe Policies} 
\label{sec:SoluRobustBaseline}
\vspace{-0.1in}
In this section, we present several different solutions to finding safe policies in MDP problems, organized from simple to more complex, and discuss their relative advantages and disadvantages. 

Before presenting any solution to this problem, let us look at the naive approach of solving the simulated MDP $\widehat{\M}$. Let $\pi_S$ be an optimal policy of $\widehat{\M}$, i.e.,~$\pi_S\in\argmax_{\pi\in\Pi_S}\rho(\pi,\widehat{\M})$ (or equivalently $\pi_S=\pi^\star_{\widehat{\M}}$). The following theorem quantifies the performance loss of this policy.

\begin{theorem}  \label{thm:perf_naive}
Let $\pi_S$ be an optimal policy of the simulator $\widehat{\M}$. Then under Assumption~\ref{asm:error}, we have
\begin{equation*}
\Phi(\pi_S) \le \frac{2 \gamma \rmax}{(1-\disc)^2}  \| e \|_\infty. \quad\quad\quad\quad\quad\quad\quad \text{\bf{\em(Appendix~\ref{appendix:ThmPerfNaive})}}
\end{equation*}
\end{theorem}
%

Unfortunately, there is no guarantee that $\pi_S$ is \emph{safe}, i.e., it performs no worse than the baseline policy $\pi_B$. Thus, deploying $\pi_S$ may lead to undesirable outcomes due to model uncertainties. In the following sections, we present methods whose solutions are guaranteed to be safe.

\vspace{-0.1in}
\subsection{Solution based on a Reward Adjusted MDP} 
\label{subsec:reward-adjusted}
\vspace{-0.1in}
In this section, we propose a method that relies on solving the MDP $\widetilde{\M}=\langle\X,\A,\widehat{r},\widehat{P},p_0,\gamma\rangle$, which is exactly the same as the simulated MDP $\widehat{\M}$, except that its reward function is \emph{adjusted} as 
\begin{equation} 
\label{eq:rhat}
\widehat{r}(x,a)=r(x,a)-\frac{\gamma R_{\max}}{1-\gamma} e(x,a), \quad \forall x\in\mathcal{X},\,\forall  a\in\mathcal{A}.
\end{equation}
The unique property of this MDP is that, when Assumption~\ref{asm:error} holds, the performance of any policy $\pi$ in $\widetilde{\M}$ is a lower-bound on its performance in the true MDP $\M^\star$, i.e.,~$\rho(\pi,\widetilde{\M})\le\rho(\pi,\M^\star)$ (see Theorem~\ref{thm:perf_reward_adjusted}). Algorithm~\ref{alg:reward-adjusted} summarizes the method for computing a policy using the reward adjusted MDP (RaMDP) $\widetilde{\M}$. It returns an optimal policy of $\widetilde{\M}$, when the performance of this policy in $\widetilde{\M}$ is better than the baseline performance $\rho(\pi_B,\M^\star)$, and returns $\pi_B$ otherwise. 

\IncMargin{1em}
\begin{algorithm}
\SetKwInOut{Input}{input}\SetKwInOut{Output}{output}
\Input{Simulated MDP $\widehat{\M}$, baseline performance $\rho(\pi_B,\M^\star)$ and error function $e$} 
\Output{Policy $\pi_{Ra}$}
$\widehat{r}(x,a) \leftarrow r(x,a)-\frac{\gamma R_{\max}}{1-\gamma} e(x,a)$ \;
$\pi_0 \leftarrow \arg \max_{\pi\in\Pi_S} \rho(\pi,\widetilde{\M})$ \;
$\rho_0 \leftarrow \rho(\pi_0,\widetilde{\M})$ \;
\text{If $\rho_0 > \rho(\pi_B,\M^\star)$}{
$\pi_{Ra} \leftarrow \pi_0$
}{
$\pi_{Ra} \leftarrow\pi_B$ \label{ln:algexp_condition}}
\Return $\pi_{Ra}$
\caption{Solution based on the RaMDP} \label{alg:reward-adjusted}
\end{algorithm}

Since the performance of any policy in $\widetilde{\M}$ is a lower-bound on its performance in $\M^\star$, it guarantes that the solution policy returned by Algorithm~\ref{alg:reward-adjusted}, $\pi_{Ra}$, performs at least as well as the baseline policy $\pi_B$. Theorem~\ref{thm:perf_reward_adjusted} shows that $\pi_{Ra}$ is a safe policy and quantifies its performance loss. 

\begin{theorem}\label{thm:perf_reward_adjusted}
Given Assumption~\ref{asm:error}, the solution $\pi_{Ra}$ of Algorithm~\ref{alg:reward-adjusted} is safe, i.e.,~$\rho(\pi_{Ra},\M^\star) \ge \rho(\pi_B,\M^\star)$. Moreover, its performance loss $\Phi(\pi_{Ra})$ satisfies 
\[ \Phi(\pi_{Ra}) \le \min\left\{\frac{2 \gamma \rmax}{(1-\disc)^2}  \| e_{\pi^\star_{\M^\star}} \|_{1,u^\star_{\M^\star}}, \Phi(\pi_{B})\right\},\]
where $u^\star_{\M^\star}$ is the normalized state occupancy frequency of the optimal policy $\pi^\star_{\M^\star}$. \hfill{\bf{\em(Appendix~\ref{appendix:ThmPerfRewAdj})}}
\end{theorem}
%
%
Note that Theorem~\ref{thm:perf_reward_adjusted} indicates that by this simple adjustment in the reward function of the simulated MDP $\widehat{\M}$, we may guarantee that our solution is safe. Moreover, it shows that the bound on the performance loss of $\pi_{Ra}$ is actually tighter than that for the solution $\pi_S$ of the simulator $\widehat{\M}$ in Theorem~\ref{thm:perf_naive}. In particular, the $L_\infty-$norm has been replaced by a weighted $L_1-$norm. In terms of computational complexity, since Algorithm \ref{alg:reward-adjusted} only requires solving a standard MDP, it can be implemented by either \emph{value iteration}, \emph{policy iteration} or \emph{linear programming}~\cite{puterman2014markov}. The corresponding complexity is therefore $O(|\mathcal A||\mathcal X|^2/(1-\gamma))$ (for value iteration)~\cite{littman1995complexity}, which is low-polynomial in $|\mathcal A|$, $|\mathcal X|$ and $1/(1-\gamma)$.  

While Algorithm \ref{alg:reward-adjusted} provides good theoretical guarantees and has low costs of computation, it may be overly conservative in many circumstances.   This is because the adjustment of the reward function is based on the assumption that there exists a state with the optimal value of $\rmax / (1-\gamma)$ and that this state is accessible from each other state with the reward $\rmax$. Since this assumption is rarely true, we propose a more adaptive formulation in the following section via RMDP methods. 

\subsection{Solution based on the Robust MDP} 
\label{subsec:robust}
\vspace{-0.1in}
Robust optimization is a standard technique to deal with model uncertainty. In this section, we propose an algorithm for finding a safe policy via solving the robust MDP (RMDP). We prove that the policy returned by this algorithm is safe and has better (sharper) worst-case guarantees. 

\begin{algorithm}
\SetKwInOut{Input}{input}\SetKwInOut{Output}{output}
\Input{Simulated MDP $\widehat{\M}$, baseline performance $\rho(\pi_B,\M^\star)$ and the error function $e$} 
\Output{Policy $\pi_R$}
\tcp{Construct the uncertainty set}
$\mathcal U(\widehat{P},e)=\left\{ P ~:~ \| P(\cdot|x,a)-\widehat{P}(\cdot|x,a)\|_1\leq e(x,a), \; \forall x,a\in\mathcal{X}\times\mathcal{A} \right\} $ \;
$ \pi_0 \leftarrow \argmax_{\pi\in\Pi_S} \min_{P\in\mathcal U(\widehat{P},e)}\rho\big(\pi,\M(P)\big)$ \;
$\rho_0 \leftarrow \min_{P\in\mathcal{U}(\widehat{P},e)} \rho\big(\pi_0,\M(P)\big)$ \; 
\text{If $\rho_0 > \rho(\pi_B,\M^\star)$}{
$\pi_R \leftarrow \pi_0$
}{
$\pi_R \leftarrow \pi_B$ \label{ln:algrob_condition}}
\Return $\pi_R$
\caption{Solution based on the RMDP} \label{alg:robust}
\end{algorithm}

The robust method is summarized in Algorithm \ref{alg:robust}. It first constructs an uncertainty set $\mathcal U(\widehat{P},e)$ using the simulator $\widehat{\M}$ and error function $e$. It then solves the resultant RMDP ($\M(P)$ with $P\in\mathcal U(\widehat{P},e)$) and returns its solution, if its worst-case performance over the uncertainty set is better than the baseline performance $\rho(\pi_B,\M^\star)$, and returns $\pi_B$ otherwise. 

Algorithm \ref{alg:robust} involves solving a ($s,a$-rectangular) RMDP. RMDPs satisfy many of the same properties as regular MDPs, such as the existence of an optimal stationary policies.  RMDPs can often be solved quite efficiently by value iteration, policy iteration, or modified policy iteration~\cite{iyengar2005robust,kaufman2013robust,nilim2005robust,wiesemann2013robust}. In practical terms, solving an RMDP can be expensive due to the need to compute the worst-case transition probability which relies on solving a convex optimization problem for each state and action in every iteration. However, the robust solution can be computed very efficiently when the uncertainty set is described, as in our formulation, in terms of an $L_1$-norm~\cite{Petrik2015}; similar results also exist for the $L_2$ norm~\cite{iyengar2005robust}.
Since implementing Algorithm \ref{alg:robust}  involves solving RMDPs, the complexity is $O(|\mathcal A||\mathcal X|^3\log(|\mathcal X|)/(1-\gamma))$ (for robust value iteration)~\cite{nilim2005robust}. It is higher than that in Algorithm \ref{alg:reward-adjusted} but is still polynomial in $|\mathcal A|$, $|\mathcal X|$ and $1/(1-\gamma)$.  
The following theorem shows that the policy $\pi_R$ is safe and quantifies its performance loss.  
\begin{theorem}\label{thm:perf_robust}
Given \ref{asm:error}, the nonempty solution $\pi_R$ of Algorithm \ref{alg:robust} is safe, i.e.,~$\rho(\pi_R,\M^\star) \ge \rho(\pi_B,\M^\star)$. Moreover, its performance loss $\Phi(\pi_R)$ satisfies
\[ \Phi(\pi_R) \le \min \left\{\frac{2 \gamma \rmax}{(1-\disc)^2}  \| e_{\pi^\star_{\M^\star}} \|_{1,u^\star_{\M^\star}},\Phi(\pi_B)\right\}, \]
where $u^\star_{\M^\star}$ is the normalized state occupancy frequency of the optimal policy $\pi^\star_{\M^\star}$. \hfill{\bf{\em(Appendix \ref{appendix:ThmPerfRobust})}}
\end{theorem}

Compared to $\pi_S$ and the bound in Theorem \ref{thm:perf_naive} on its performance loss, Theorem \ref{thm:perf_robust} indicates that the policy $\pi_R$ returned by Algorithm \ref{alg:robust} is safe and has a smaller bound on its performance loss. In particular, the bound depends only on a weighted $L_1-$norm of the errors for the optimal policy, instead of the $L_\infty-$norm over all policies. 

While the complexity of Algorithm \ref{alg:robust} is higher than that of Algorithm \ref{alg:reward-adjusted} (because solving RMDPs is more complex than standard MDP), Theorem \ref{thm:perf_robust} does not show any advantage for $\pi_R$ over the policy $\pi_{RA}$ returned by Algorithm \ref{alg:reward-adjusted}, neither in terms of safety nor in terms of the bound on its performance loss (Theorem \ref{thm:perf_reward_adjusted}). This arises the question that why should one use the more complex Algorithm \ref{alg:robust} in place of Algorithm \ref{alg:reward-adjusted}. Proposition \ref{prop:robust-vs-rewadj} provides an answer to this question and shows whenever Algorithm \ref{alg:reward-adjusted} returns a safe policy, so does Algorithm \ref{alg:robust}, while the converse is not necessarily true. This implies with extra computational complexity, the conservatism of safe policy search decreases.

\begin{proposition}
\label{prop:robust-vs-rewadj}
Given \ref{asm:error}, for each policy $\pi$, we have
\[ \min_{P\in\mathcal{U}(\widehat{P},e)} \return\big(\pi,\M(P)\big) \ge \return(\pi,\widetilde{\M})~, \]
where $\widetilde{M}$ is defined in Section \ref{subsec:reward-adjusted}.\hfill{\bf{\em(\ref{appendix:prop-robust-reward})}}
\end{proposition}
%

Note that the bound in $\pi_R$ is based on the worst-case transition probabilities in $\mathcal U(\widehat{P},e)$. While this approach guarantees to yield a safe policy, it may still be too conservative as solving a RMDP considers finding a safe policy under the worst-case scenario. Next, we investigate an alternative approach to return a safe but less conservative policy.

\subsection{Solution based on an Augmented Robust MDP}
\label{subsec:aug_robust}
\vspace{-0.1in}
As discussed at the end of Section~\ref{subsec:robust}, the goal in this section is to develop a new method that combines simulated and RMDPs, and reduces the conservatism of {\em safe policy search} compared to Algorithm~\ref{alg:robust}. We start this section by considering the following constraint optimization problem that finds a policy that maximizes the performance in the simulator and satisfies the safety constraint:
%
%
\begin{equation}
\label{eq:1_sim_real}
\max_{\pi\in\Pi_H}\quad \rho(\pi,\widehat{\M}), \quad\quad\quad\quad\quad \text{subject to}\quad \rho(\pi,\M^\star)\geq\rho(\pi_B,\M^\star),
\end{equation}
where $\Pi_H$ is the general set of history-based policies. To solve~\eqref{eq:1_sim_real}, we employ the Lagrangian relaxation procedure~\cite{bertsekas1999nonlinear} to convert it to the following unconstrained optimization problem:
%
%
\begin{equation}
\label{eq:2_sim_real}
\max_{\pi\in\Pi_H}\min_{\lambda\geq 0}L(\pi,\lambda) :=\rho(\pi,\widehat{\M})+\lambda \Bigl( \rho(\pi,\M^\star)-\rho(\pi_B,\M^\star) \Bigr),
\end{equation}
where $\lambda$ is the Lagrange multiplier. Unfortunately, solving the optimization problem~\eqref{eq:1_sim_real} is impossible, since the true MDP $\M^\star$ is unknown. Before describing how we tackle this issue, let us define a few terms and quantities. For the simulated MDP $\widehat{\M}$ and any MDP $\M(P)$ that is only different with $\widehat{\M}$ in its transition probability function, and for any $\lambda_1,\lambda_2\geq 0$, we define the {\em augmented} MDP $\M^A_{\lambda_1,\lambda_2}=\langle\X\times\X,\A,r^A,P^A,p^A_0,\gamma\rangle$, where $r^A(x,y,a)=\lambda_1r(x,a)+\lambda_2r(y,a)$, $P^A(x',y'|x,y,a)=P(x'|x,a)\widehat{P}(y'|y,a)$, and $p_0^A(x,y)=p_0(x)p_0(y)$. Here we can think of $x$'s and $y$'s as the states evolved according to the MDPs $\M(P)$ and $\widehat{\M}$, respectively. We denote by $\Pi_S^A$ the set of stationary Markovian policies over this augmented MDP. Using the augmented MDP, for any policy $\pi\in\Pi_H$, we define the Lagrangian function 
\begin{equation}
\label{eq:lag}
L_P(\pi,\lambda)=\rho\big(\pi,\M^A_{\lambda,1}(P)\big)-\lambda\mathcal \rho(\pi_B,\M^\star).
\end{equation}
Using the above definitions, the Lagrangian function in~\eqref{eq:2_sim_real} may be easily written as $L(\pi,\lambda)=L_{P^\star}(\pi,\lambda)$. From Theorem 3.6 in~\cite{altman1999constrained}, it can be easily shown that for any transition probability function $P\in\mathcal U(\widehat P,e)$, the following strong duality holds:
\begin{equation}
\label{eq:strong-dual}
\max_{\pi\in\Pi_H}\min_{\lambda\geq 0} L_P(\pi,\lambda)=\max_{\pi\in\Pi_S^A}\min_{\lambda\geq 0} L_P(\pi,\lambda)=\min_{\lambda\geq 0}\max_{\pi\in\Pi_S^A}  L_P(\pi,\lambda).
\end{equation}
Setting $P=P^\star$, \eqref{eq:strong-dual} implies that we can replace the class of history-based policies $\Pi_H$ with the set of stationary Markovian policies over the augmented MDP $\Pi^A_S$ in problem~\eqref{eq:2_sim_real}. This further means that $\Pi^A_S$ is the class of \emph{dominating} policies for the optimization problem~\eqref{eq:1_sim_real}. 

Recall that the duality theory~\cite{bertsekas1999nonlinear} indicates that if the dual Lagrangian problem is bounded, the primal Lagrangian problem is always feasible. This means that if we can find a lower-bound for the dual Lagrangian problem $\min_{\lambda\geq 0}\max_{\pi\in\Pi_S^A}L(\pi,\lambda)$ (note that $L(\pi,\lambda)=L_{P^\star}(\pi,\lambda)$ and $P^\star\in\mathcal U(\widehat P,e)$), the corresponding policy will be feasible for the constraint in~\eqref{eq:1_sim_real}, which itself means that it is safe. This motivates us to find a saddle-point for the (augmented) robust optimization problem
\begin{equation}
\label{problem:LB}
\min_{\lambda\geq 0}\;\max_{\pi\in\Pi_S^A}\;\min_{P\in\mathcal U(\widehat{P},e)} L_P(\pi,\lambda).
\end{equation}
Note that compared to the Lagrangian function $L(\pi,\lambda)$, in~\eqref{problem:LB}, we have replaced the true transition probability $P^\star$ with the worst-case transition probability over the uncertainty set $\mathcal U(\widehat P,e)$. The reason for finding a saddle-point of~\eqref{problem:LB} is because the solution is a lower-bound for the dual Lagrangian:   
\begin{equation*}
\min_{\lambda\geq 0}\max_{\pi\in\Pi_S^A}\min_{P\in\mathcal U(\widehat{P},e)} L_P(\pi,\lambda) \stackrel{\text{(a)}}{=} \min_{P\in\mathcal U(\widehat{P},e)}\min_{\lambda\geq 0}\max_{\pi\in\Pi_S^A} L_P(\pi,\lambda) \stackrel{\text{(b)}}{\ge} \min_{\lambda\geq 0}\max_{\pi\in\Pi_S^A} L(\pi,\lambda),
\end{equation*}
{\bf (a)} Theorem~1 in~\cite{nilim2005robust} shows that strong duality holds in $(\X\times\A)$-rectangular robust optimization problems, i.e.,~
\[
\min_{P\in\mathcal U(\widehat{P},e)}\max_{\pi\in\Pi_S^A} L_P(\pi,\lambda)=\max_{\pi\in\Pi_S^A}\min_{P\in\mathcal U(\widehat{P},e)} L_P(\pi,\lambda),
\] \\
{\bf (b)} This is from the fact that $L(\pi,\lambda)=L_{P^\star}(\pi,\lambda)$ and $P^\star\in\mathcal U(\widehat{P},e)$.

Therefore, if we find a saddle point $(\pi_0,\lambda^\star)$ of the (augmented) robust optimization problem~\eqref{problem:LB}, then the corresponding policy $\pi_0$ is safe. Given the above observations, we now present Algorithm~\ref{alg:robust_aug} and prove in Theorem~\ref{thm:perf_robust_aug} that the policy returned by this algorithm is safe and quantify its performance loss. On Line~\ref{ln:aug_rob_opt} of Algorithm~\ref{alg:robust_aug}, we use the conventional sub gradient descent approach to solve for a saddle-point. In this approach, we first fix the Lagrange multiplier and solve for an optimal stationary policy and then optimize for the Lagrangian multiplier (which is a convex optimization problem). These two steps are repeated until the solution converges to a saddle point. More details of this procedure can be found in Appendix~\ref{appendix:saddle_point}. Regarding to the computation cost of this approach, since each sub-gradient descent step involves a RMDP, similar to Algorithm \ref{alg:robust} it has complexity of $O(|\mathcal A||\mathcal X|^3\log(|\mathcal X|)/(1-\gamma))$ for robust value iteration. Thus the total complexity of Algorithm \ref{alg:robust_aug} is $O(|\mathcal A||\mathcal X|^3\log(|\mathcal X|)/(1-\gamma)(1/\sqrt{K}))$ \cite{boyd2004convex}, where $K$ is the number of iteration of sub-gradient descent and $O(1/\sqrt{K})$ is the standard convergence rate for first order methods in convex optimization. 

Regarding to the augmented Markovian policy $\pi_{0}$ in Algorithm \ref{alg:robust_aug}, since $\pi_0$ requires state information from both the uncertain MDP $\M(P)$ (with state $X_t$) and simulated MDP $\widehat\M$ (with state $Y_t$) . Therefore implementing this policy in real-world (i.e. $P=P^\star$) requires real-time state trajectories from the online simulator as well. This inevitably increases the complexity of implementation.
 
\begin{algorithm}
\SetKwInOut{Input}{input}\SetKwInOut{Output}{output}
\Input{Simulated MDP $\widehat{\M}$, baseline performance $\rho(\pi_B,\M^\star)$ and the error function $e$}
\Output{Policy $\pi_{AR}$}
Construct the uncertainty set $\mathcal U(\widehat{P},e)$ and augmented MDP ${\M}^A_{\lambda,1}(P),\;\forall P\in\mathcal U(\widehat{P},e)$ \;
Solve $\min_{\lambda\geq 0}\max_{\pi\in\Pi_S^A}\min_{P\in\mathcal U(\widehat{P},e)}L_P(\pi,\lambda)$ for a saddle-point $(\pi_0,\lambda^\star)$ \; \label{ln:aug_rob_opt}
\text{If a saddle point solution $(\pi_0,\lambda^\star)$ exists}{
$\pi_{AR} \leftarrow \pi_0$
}{
$\pi_{AR} \leftarrow\pi_B$}
\Return $\pi_{AR}$
\caption{Solution based on the Augmented RMDP} \label{alg:robust_aug}
\end{algorithm}

\begin{theorem}\label{thm:perf_robust_aug}
Given Assumption~\ref{asm:error}, the nonempty solution $\pi_{AR}$ of Algorithm~\ref{alg:robust_aug} is safe, i.e.,~$\rho(\pi_{RS},\M^\star) \ge \rho(\pi_B,\M^\star)$. Moreover, its performance loss $\Phi(\pi_{AR})$ satisfies 
\begin{equation*}
\Phi(\pi_{AR}) \le\min\left\{\frac{2 \gamma \rmax}{(1-\disc)^2}  \| e_{\pi^\star_{\M^\star}} \|_{1,u^\star_{\M^\star}},\Phi(\pi_B)\right\},
\end{equation*}
where $u^\star_{\M^\star}$ is the normalized state occupancy frequency of the optimal policy $\pi^\star_{\M^\star}$. \hfill{\bf{\em(Appendix~\ref{appendix:robust_aug_perf})}}
\end{theorem}
%

Similar to~Section \ref{subsec:robust}, compared to $\pi_S$ and the bound in~Theorem \ref{thm:perf_naive} on its performance loss,~Theorem \ref{thm:perf_robust_aug} indicates that the policy $\pi_{AR}$ returned by~Algorithm \ref{alg:robust_aug} is safe and has a tighter bound on its performance loss. However, while~Algorithm \ref{alg:robust_aug} has higher computational complexity than~Algorithm \ref{alg:reward-adjusted} and~Algorithm \ref{alg:robust},~Theorem \ref{thm:perf_robust_aug} does not show any advantage for its returned policy $\pi_{AR}$, over those returned by~Algorithm \ref{alg:reward-adjusted} and~Algorithm \ref{alg:robust}, $\pi_{RA}$ and $\pi_R$. This raises a question similar to that in~Section \ref{subsec:robust} that why should we use ~Algorithm \ref{alg:robust_aug} instead of~Algorithm \ref{alg:robust} then? Proposition~\ref{prop:robust-vs-augmented_robust} provides an answer to this question and shows whenever~Algorithm \ref{alg:robust} returns a safe policy, so does~Algorithm \ref{alg:robust_aug}, while the converse is not always true. This again resonates with the fact that with extra computational complexity, we reduce conservatism of safe policy search.

\begin{proposition}
\label{prop:robust-vs-augmented_robust}
Given Assumption~\ref{asm:error}, if $\rho(\pi_B,\M^\star)\leq\min_{\mathcal P\in\mathcal U(\widehat{P})}\rho\big(\pi_R,\M(P)\big)$, then $\min_{\mathcal P\in\mathcal U(\widehat{P})}L_P(\pi_{AR},\lambda^\star)$ is lower-bounded. This means that if Algorithm~\ref{alg:robust} returns a safe policy other than $\pi_B$, so does Algorithm~\ref{alg:robust_aug}. \hfill{\bf{\em(Appendix~\ref{appendix:robust_aug_feasibility})}}
\end{proposition}

\newcommand{\polrobi}{\pi_{\text{I}}}
\newcommand{\vbase}{v_{\text{B}}}
\newcommand{\qbase}{v_{\text{B}}}
\newcommand{\uset}[1]{\mathcal{U}(#1)}
\subsection{Combining Robust and Baseline Policies} \label{sec:combined}
\vspace{-0.1in}
While the robust solution described in Algorithm \ref{alg:robust} (or Algorithm \ref{alg:robust_aug}) is less conservative than Algorithm \ref{alg:reward-adjusted} it may nevertheless be overly restrictive. This is because the proposed improved policy $\pi_0$ is evaluated for the worst-case realization of the transition probability in $\mathcal{U}(\widehat{P},e)$ while the return of the baseline policy is with respect to $P\opt$. As a result, a candidate policy $\pi_0$ can be rejected even if for any realization $P\in\mathcal{U}(\widehat{P},e)$ the policy is better than the baseline. The left example in Figure \ref{fig:conservative} depicts the case in which the robust solution will be too restrictive.

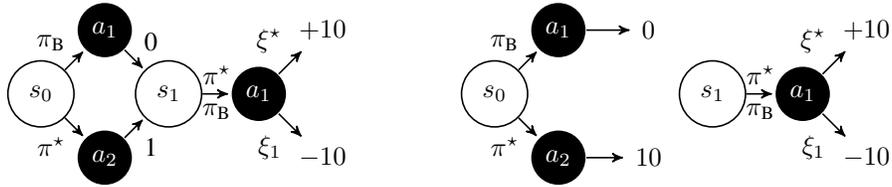
\begin{figure}[h]
\centering
\begin{tikzpicture}[->,>=stealth',shorten >=1pt,auto,node distance=1.2cm,semithick]
	\tikzstyle{action}= [circle,fill=black,text=white]

	\node[state] 		 (s1)                    														{$s_0$};
	\node[action]      (a1)         [above right of=s1]     	{$a_1$};
	\node[action]      (a2)         [below right of=s1]    {$a_2$};

	\node[state] (s2) [below right of=a1] {$s_1$};

	\node[action] (a11) [right of=s2,xshift=0mm] {$a_1$};

 	\node (o1) [above right of=a11] {$+10$};
 	\node (o2) [below right of=a11] {$-10$};

 	\path (s1) edge node [above left]{$\pi_{\text{B}}$}  (a1) ;
	\path (a1) edge node[above right] {0} (s2);
 	\path (a11) edge node [above left]{$\xi^\star$} (o1);
 	\path (a11) edge node [below left]{$\xi_1$} (o2);
 
 	\path (s1) edge node [below left]{$\pi^\star$}  (a2) ;
	\path (a2) edge node[below right] {1} (s2);

	\path (s2) edge node[above] {$\pi^\star$} node[below] {$\pi_{\text{B}}$} (a11);
\end{tikzpicture}\qquad\qquad
\begin{tikzpicture}[->,>=stealth',shorten >=1pt,auto,node distance=1.2cm,semithick]
	\tikzstyle{action}= [circle,fill=black,text=white]

	\node[state] 		 (s1)                    														{$s_0$};
	\node[action]      (a1)         [above right of=s1]     	{$a_1$};
	\node[action]      (a2)         [below right of=s1]    {$a_2$};

 	\node (o11) [right of=a1] {$0$};
 	\node (o12) [right of=a2] {$10$};

	\node[state] (s2) [below right of=o11] {$s_1$};

	\node[action] (a11) [right of=s2,xshift=0mm] {$a_1$};

 	\node (o1) [above right of=a11] {$+10$};
 	\node (o2) [below right of=a11] {$-10$};

 	\path (s1) edge node [above left]{$\pi_{\text{B}}$}  (a1) ;

	\path (a1) edge  (o11);
	\path (a2) edge  (o12);

 	\path (a11) edge node [above left]{$\xi^\star$} (o1);
 	\path (a11) edge node [below left]{$\xi_1$} (o2);

 	\path (s1) edge node [below left]{$\pi^\star$} (a2) ;

	\path (s2) edge node[above] {$\pi^\star$} node[below] {$\pi_{\text{B}}$} (a11);
\end{tikzpicture}
\caption{Left: Example in which the policy returned by~Algorithm \ref{alg:robust} is too restrictive. Right: Example in which the policy returned by Algorithm~\ref{alg:robust} performs better than the baseline policy in some states but not in others.} \label{fig:conservative}
\end{figure}
An additional limitation from Algorithm \ref{alg:reward-adjusted}, \ref{alg:robust} and \ref{alg:robust_aug} is that when the evaluation criterion of the computed policy fails, then the baseline policy is not improved at all. The restriction of these approaches is illustrated by the following counter example:

Consider a simple case when the MDP $\M$ is composed of two separate MDPs $\M_1$ and $\M_2$ and suppose that the estimated model of $\M_1$ is very good, while the model of $\M_2$ is quite imprecise. Also assume that the initial distribution is uniformly distributed between an initial state in $\M_1$ and an initial state in $\M_2$. Intuitively, the best solution would use the optimized policy for $\M_1$, which has a precise model, and use the baseline policy for $\M_2$, which has an imprecise model. However, Algorithm \ref{alg:robust} simply returns the baseline policy for both components because the return in $\M_2$ potentially reduces the quality of the robust solution. This phenomenon is illustrated in the right example of Figure \ref{fig:conservative} for which it would beneficial to combine the baseline policy with the optimized one instead of returning either one.

The above limitations can be solved by modifying the robust optimization problem. Intuitively, the objective is to find a policy that maximizes improvement over the baseline for any plausible transition probabilities. Equivalently this reduces to a distributionally robust optimization (DRO):
\begin{equation} \label{eq:objective_robust_interleave}
\polrobi \in \arg \max_\pi \min_{P \in \mathcal{U}(\widehat{P},e)} \Bigl( \return(\pi, P) - \return(\pi_B, P)\Bigr).
\end{equation}
Compared with Algorithm \ref{alg:reward-adjusted}, \ref{alg:robust} and \ref{alg:robust_aug}, the first appealing fact of this approach is that the solution policy is always safe because $\pi_B$ is feasible to \eqref{eq:objective_robust_interleave}. The second appealing fact is that it only requires the knowledge of baseline policy $\pi_B$, without the baseline performance $\rho(\pi_B,P^\star)$. Nevertheless, \eqref{eq:objective_robust_interleave} is a DRO problem which in general could be NP-hard to solve \cite{delage2010distributionally}. Here the computational methods used to solve this problem are beyond the scope of this paper; we approximate the solution using a heuristic iterative algorithm based on value iteration. The additional complexity of this formulation again corroborates with the phenomenon that additional computational complexity reduces conservatism in safe policy search. Similar to the other three formulations, the following theorem states the safety of the computed policy and describes its performance loss.
\begin{theorem}\label{thm:perf_robust_interleave}
Given that Assumption \ref{asm:error} is satisfied, then a solution $\polrobi$ to \eqref{eq:objective_robust_interleave} is safe, i.e., $\rho(\polrobi,P\opt) \ge \rho(\polbase,P\opt)$. Moreover, the performance loss of $\polrobi$ satisfies: 
\[ \Phi(\polrobi) \le \min \left\{ \frac{2 \gamma \rmax}{(1-\disc)^2}  \Big( \| e_{\pi^\star_{\M^\star}} \|_{1,u^\star_{\M^\star}}  + \| e_{\polbase} \|_{1,u_{\text{B},\M^\star}} \Big),  \Phi(\polbase) \right\} ~, \]
where $u^\star_{\M^\star}, u_{\text{B},\M^\star}$ are the normalized state occupancy frequencies of the optimal and base policies $\pi^\star_{\M^\star},\polbase$ respectively. Also, this bound is tight. \hfill{\bf{\em(Appendix \ref{appendix:robust_aug_interleave})}}
\end{theorem}

\section{Numerical Comparison} \label{sec:experiments}
\vspace{-0.1in}
In this section, we numerically evaluate the proposed methods on a synthetic benchmark MDP.  The benchmark problem loosely models customer interactions with an online system. The four available actions influence the user behavior along two dimensions. Rewards, which represent user satisfactions, vary only along the first dimension. The second dimension influences only transition probabilities. To simulate a realistic source of a baseline policy, we construct it to be optimal when the second dimension of the MDP is ignored. The simulator is constructed directly from the empirical transition probabilities. The transition error $e$ is based on the sampling bounds in Section \ref{subsec:sampling_bounds} and decreases with a square root of the number of samples.  
\begin{figure}
\centering
\includegraphics[width=0.5\linewidth]{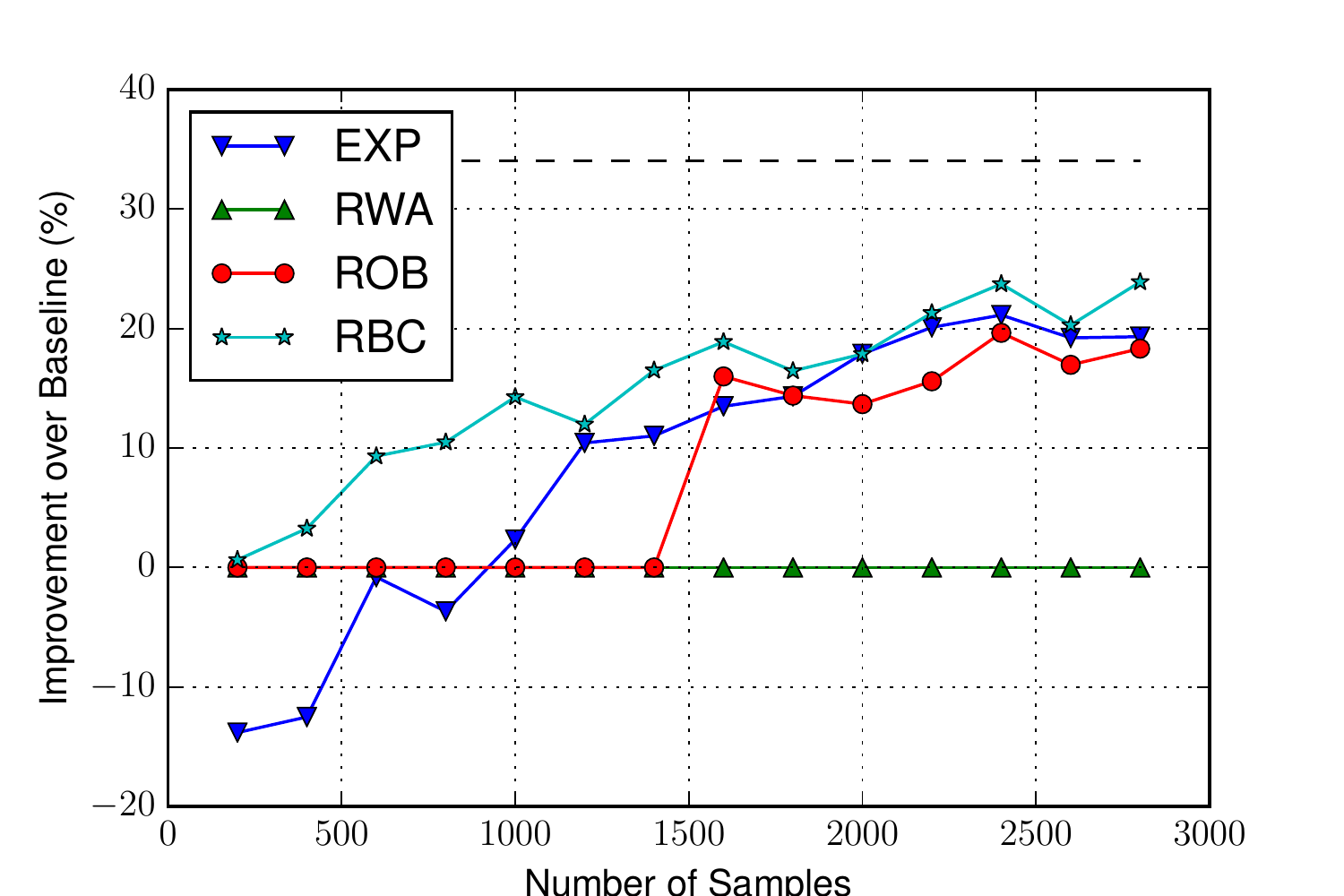}
\caption{Improvement in return over the baseline policy for the proposed methods.} \label{fig:result}
\end{figure}
Figure \ref{fig:result} depicts the percentage improvement in total return over the baseline policy as a function of the overall number of samples used in constructing the simulator. The methods used in the comparison are as follows. The dashed line shows the return of the optimal policy. EXP stands for the standard MDP method with the expected return objective. For a small number of samples, this standard method does significantly worse than the baseline policy. RWA stands for the method in Algorithm \ref{alg:reward-adjusted}, which leads to a safe policy but, as expected, is overly conservative. ROB stands for the robust method in Algorithm \ref{alg:robust}, which also guarantees the safety of returned policies, but is much less conservative than RWA. Finally, RBC is the algorithm described in Section \ref{sec:combined}. RBC optimizes the policy in states with many samples and falls back onto the baseline policy otherwise. The combined policy of RBC is not only safe, but also significantly improves on the baseline policy even when the number of samples is small. 

\section{Conclusion} \label{sec:conclusion}
\vspace{-0.1in}
In this paper we presented four model based safe policy search methods and analyzed their performance. Ranging from computational complexity to conservatism, our approaches provide a full gamut of tools to design good policies offline that match baseline performance. To the best of our knowledge, this line of work is \emph{novel} 
in the RL community. Similar approaches in the model-free setup can be found in \cite{ghavamzadeh2012conservative,kakade2002approximately,thomas2015high}, where safe policy evaluation takes place during exploration. 

On the technical side, an important future direction is to compare the performance of policies generated by different safe policy search algorithms and to explicitly study the solution algorithm in Section \ref{sec:combined}. On the experimental side, future work includes running advanced simulations in realistic domains such as battery charging/discharging control for smart grid systems.

\begin{small}
\bibliography{LTV}
\bibliographystyle{plainnat}
\end{small}

\newpage
\appendix
\onecolumn
\section{Proofs}
\subsection{Sampling Bounds} \label{subsec:sampling_bounds}

\cite{wiesemann2013robust} showed that $L_1-$deviation of the empirical distribution from the true distribution over $m$ distinct events from $n$ samples is bounded as
\begin{equation}
\label{eq:L1Bound}
\mathbb{P}\left\{||P(\cdot)-\widehat{P}(\cdot)||_1\geq \epsilon\right\} \leq (2^m-2)\exp\left(-\frac{n\epsilon^2}{2}\right).
\end{equation}
Now consider a fixed state-action pair $(x, a)$ and assume that the transition probability $\widehat{P}(\cdot|x,a)$ has been estimated using $N(x,a)$ visits to $(x,a)$. The random event for the transition probability estimate is the state to which the system transits. In this case, $m=|\mathcal{X}|$ and using~\eqref{eq:L1Bound} we may write  
\begin{equation*}
||P(\cdot|x,a)-\widehat{P}(\cdot|x,a)||_1\leq \sqrt{\frac{2}{N(x,a)}\log(\frac{2^{|\mathcal{X}|}-2}{\delta})},
\end{equation*}
with probability at least $1-\delta$. Thus, by setting $e(x,a)=\sqrt{\frac{2}{N(x,a)}\log(\frac{|\mathcal{X}||\mathcal{A}|(2^{|\mathcal{X}|}-2)}{\delta'})}$, we can guarantee that $\mathbb{P}\left\{P^\star\notin\mathcal{U}(\widehat{P})\right\}\leq\delta$. 
\subsection{Proof of Lemma~\ref{cor:policy_tran_error}}
\label{appendix:LemmaPolTranErr}

Since the return of a policy $\pi$ is the product of the initial state distribution and the value function of the policy, i.e.,~$\return(\pi,\M)=p_0^\top V^\pi_\M$, Lemma~\ref{cor:policy_tran_error} is a direct consequence of the following lemma.  

\begin{lemma} 
\label{lem:bound_transitions}
Consider two MDPs $\M_1$ and $\M_2$ that are only different in their transition probability functions $P_1$ and $P_2$, and reward functions $r_1$ and $r_2$. Let $\pi_1$ be a policy in $\M_1$ and $\pi_2$ be a policy in $\M_2$. Under the assumption that for any state $x\in\X,\;\| P_1^{\pi_1}(\cdot | x) - P_2^{\pi_2}(\cdot | x) \|_1 \le g_x$, we have
\begin{equation*}
(\eye - \disc P_1^{\pi_1})^{-1} \Bigl( r_1^{\pi_1} - r_2^{\pi_2} - \frac{\gamma \rmax}{1-\gamma} g \Bigr) \le V_{\M_1}^{\pi_1} - V_{\M_2}^{\pi_2} \le (\eye - \disc P_1^{\pi_1})^{-1} \Bigl( r_1^{\pi_1} - r_2^{\pi_2} + \frac{\gamma \rmax}{1-\gamma} g  \Bigr),
\end{equation*}
where $g$ is the vector of $g_x$'s. Moreover, the above inequalities are tight.
\end{lemma}
\begin{proof}
The difference between the two value functions can be written as follows:
\begin{align*}
V_{\M_1}^{\pi_1} - V_{\M_2}^{\pi_2} &= r_1^{\pi_1} + \gamma P_1^{\pi_1} V_{\M_1}^{\pi_1} - r_2^{\pi_2} - \gamma P_2^{\pi_2}V_{\M_2}^{\pi_2} \\
&= r_1^{\pi_1} + \gamma P_1^{\pi_1} V_{\M_1}^{\pi_1} - r_2^{\pi_2} - \gamma P_2^{\pi_2} V_{\M_2}^{\pi_2} + \gamma P_1^{\pi_1}V_{\M_2}^{\pi_2} - \gamma P_1^{\pi_1}V_{\M_2}^{\pi_2} \\
&= (r_1^{\pi_1} - r_2^{\pi_2}) + \gamma P_1^{\pi_1}(V_{\M_1}^{\pi_1} - V_{\M_2}^{\pi_2}) + \gamma (P_1^{\pi_1} - P_2^{\pi_2}) V_{\M_2}^{\pi_2} \\
&= (\eye - \gamma P_1^{\pi_1})^{-1} \left[r_1^{\pi_1} - r_2^{\pi_2} + \gamma (P_1^{\pi_1} - P_2^{\pi_2})V_{\M_2}^{\pi_2}\right].
\end{align*}
Now using the Holder's inequality, for any $x\in\X$, we have
\begin{equation*}
| \big(P_1^{\pi_1}(\cdot|x) - P_2^{\pi_2}(\cdot|x)\big)\tr V_{\M_2}^{\pi_2} | \le \| P_1^{\pi_1}(\cdot|x) - P_2^{\pi_2}(\cdot|x) \|_1 \| V_{\M_2}^{\pi_2} \|_\infty \le  g_x \| V_{\M_2}^{\pi_2} \|_\infty \le  g_x \frac{R_{\max}}{1-\gamma}.
\end{equation*}
The proof follows by uniformly bounding $(P_1^{\pi_1} - P_2^{\pi_2})V_{\M_2}^{\pi_2}$ from the above inequality and from the monotonicity of $(\eye - \gamma P_1^{\pi_1})^{-1}$.
\end{proof}


\subsection{Proof of Theorem \ref{thm:perf_naive}}
\label{appendix:ThmPerfNaive}

\begin{proof}
From Lemma~\ref{cor:policy_tran_error} with $\pi_1=\pi_2=\pi_S$, $\M_1=\M^\star$, and $\M_2=\widehat{\M}$, we have
\begin{equation*}
\return(\pi_S,\widehat{\M}) - \frac{\gamma\rmax}{1-\gamma}p_0^\top(\eye-\gamma P^\star_{\pi_S})^{-1}e_{\pi_S}\leq \return(\pi_S,\M^\star).
\end{equation*}
Thus, we may write
\begin{align*}
\Phi(\pi_S)\stackrel{\Delta}{=}\return(\pi^\star_{\M^\star},\M^\star) - \return(\pi_S,\M^\star) &\le \return(\pi^\star_{\M^\star},\M^\star) - \return(\pi_S,\widehat{\M}) + \frac{\gamma \rmax}{1-\disc} \indist\tr (\eye - \disc P^\star_{\pi_S})^{-1} e_{\pi_S} \\
&\stackrel{\text{(a)}}{\le} \return(\pi^\star_{\M^\star},\M^\star) - \return(\pi^\star_{\M^\star},\widehat{\M}) + \frac{\gamma \rmax}{1-\disc} \indist\tr (\eye - \disc P^\star_{\pi_S})^{-1} e_{\pi_S} \\
&\stackrel{\text{(b)}}{\le} \frac{\gamma \rmax}{1-\disc} \indist\tr \left[(\eye - \disc P^\star_{\pi^\star_{\M^\star}})^{-1}e_{\pi^\star_{\M^\star}} + (\eye - \disc P^\star_{\pi_S})^{-1}e_{\pi_S}\right] \\
&\stackrel{\text{(c)}}{\le} \frac{2\gamma \rmax}{(1-\disc)^2}\| e \|_\infty.
\end{align*}
{(a)} Comes from the optimality of $\pi_S$ in $\widehat{\M}$.\\
{(b)} This is the application of Lemma~\ref{cor:policy_tran_error} with policy $\pi_1=\pi_2=\pi^\star_{\M^\star}$, $\M_1=\M^\star$, and $\M_2=\widehat{\M}$. \\
{(c)} This is from the fact that for any policy $\pi$, we have $\| \indist\tr (\eye - \disc P^\star_\pi)^{-1} \|_1 = 1/(1-\gamma)$, and from the application of the Holder's inequality.
\end{proof}


\subsection{Proof of Theorem \ref{thm:perf_reward_adjusted}}
\label{appendix:ThmPerfRewAdj}

\begin{proof}
To prove the safety of $\pi_{Ra}$ and bound its performance loss, we need to upper and lower bound the difference between the performance of any policy $\pi$ in the true MDP $\M^\star$ and its performance in $\widetilde{M}$, i.e.,~$\return(\pi,\M^\star)-\return(\pi,\widetilde{\M})$. These upper and lower bounds are obtained by applying Lemma~\ref{lem:bound_transitions} with $\pi_1=\pi_2=\pi$, $\M_1=\M^\star$, and $\M_2=\widetilde{\M}$ as follows
\begin{equation}
\label{eq:AppC1}
\return(\pi,\M^\star) - \return(\pi,\widetilde{\M}) \ge p_0\tr (\eye - \gamma P^\star_\pi)^{-1} \left(r^\pi - \widehat{r}^\pi - \frac{\gamma \rmax}{1-\gamma} e_\pi \right) \ge 0,
\end{equation}
where the second inequality in~\eqref{eq:AppC1} follows from the definition of the adjusted reward function $\widehat{r}$, and the fact that $(\eye - \gamma P^\star_\pi)^{-1}$ is monotone and $p_0$ is non-negative. Similarly, the upper-bound is
\begin{equation}
\label{eq:exp_proof_lower}
\return(\pi,\M^\star) - \return(\pi,\widetilde{\M}) \le \frac{2\gamma \rmax}{1-\gamma} p_0\tr (\eye - \gamma P^\star_\pi)e_\pi = \frac{2\gamma \rmax}{(1-\gamma)^2} \, \| e_\pi \|_{1,u^\pi_{\M^\star}},    
\end{equation}
where $u^\pi_{\M^\star}= (1-\gamma) p_0^\top(\eye - \gamma P^\star_{\pi^\star})^{-1}$ is the normalized state occupancy frequency of policy $\pi$ in the true MDP $\M^\star$.

To prove the safety of the returned policy $\pi_{Ra}$, consider the two cases on Line~\ref{ln:algexp_condition} of Algorithm~\ref{alg:reward-adjusted}. When the condition is satisfied, we have $\return(\pi_B,\M^\star) < \return(\pi_0,\widetilde{\M}) \le \return(\pi_0,\M^\star)$, where the second inequality comes from~\eqref{eq:AppC1}, and thus, the policy $\pi_{Ra}=\pi_0$ is safe. When the condition is violated, then $\pi_{Ra}$ is simply $\pi_B$, which is safe by definition.


To derive a bound on the performance loss of the returned policy $\pi_{Ra}$, consider also the two cases on Line~\ref{ln:algexp_condition} of Algorithm~\ref{alg:reward-adjusted}. When the condition is satisfied, using~\eqref{eq:AppC1}, we have
\begin{equation*}
\Phi(\pi_{Ra}) = \return(\pi^\star_{\M^\star},\M^\star) - \return(\pi_0,\M^\star) \le \return(\pi^\star_{\M^\star},\M^\star) - \return(\pi_0,\widetilde{\M}),
\end{equation*}
and when the condition is violated, we have 
\begin{equation*}
\Phi(\pi_{Ra}) = \return(\pi^\star_{\M^\star},\M^\star) - \return(\pi_B,\M^\star).
\end{equation*}
Since the condition is satisfied on Line~\ref{ln:algexp_condition} of Algorithm~\ref{alg:reward-adjusted} when $\return(\pi_0,\widetilde{\M})>\return(\pi_B,\M^\star)$, we may write 
\begin{equation*}
\Phi(\pi_{Ra}) \le \min\left\{\return(\pi^\star_{\M^\star},\M^\star) - \return(\pi_0,\widetilde{\M})\;,\;\return(\pi^\star_{\M^\star},\M^\star) - \return(\pi_B,\M^\star)\right\}.
\end{equation*}
Note that we may write the following inequalities for the first term in the minimum
\begin{equation*}
\return(\pi^\star_{\M^\star},\M^\star) - \return(\pi_0,\widetilde{\M}) \stackrel{\text{(a)}}{\leq} \return(\pi^\star_{\M^\star},\M^\star) - \return(\pi^\star_{\M^\star},\widetilde{\M}) \stackrel{\text{(b)}}{\le} \frac{2\gamma \rmax}{(1-\gamma)^2} \, \| e_{\pi^\star_{\M^\star}} \|_{1,u^\star_{\M^\star}},
\end{equation*}
where {\bf (a)} follows from $\pi_0$ being an optimal policy of MDP $\widetilde{\M}$, {\bf (b)} is from~\eqref{eq:exp_proof_lower} with $\pi=\pi^\star_{\M^\star}$, and $u^\star_{\M^\star}$ is the normalized state occupancy frequency of the optimal policy $\pi^\star_{\M^\star}$. This proves the theorem.
\end{proof}

\subsection{Proof of Theorem \ref{thm:perf_robust}}
\label{appendix:ThmPerfRobust}

\begin{proof}
To prove the safety of $\pi_R$ and bound its performance loss, we need to upper and lower bound the difference between the performance of any policy $\pi$ in the true MDP $\M^\star$ and its worst-case performance, $\min_{P\in\mathcal{U}(\widehat{P},e)} \return\big(\pi,\M(P)\big)$. Since $P^\star\in\mathcal{U}(\widehat{P},e)$ from Assumption~\ref{asm:error}, we have
\begin{equation} 
\label{eq:rob_proof_upper}
\min_{P\in\mathcal{U}(\widehat{P},e)}\return\big(\pi,\M(P)\big) \le \return(\pi,\M^\star).
\end{equation} 
Now let $\bar{P}$ be the minimizer in $\min_{P\in\mathcal{U}(\widehat{P},e)}\return\big(\pi,\M(P)\big)$. The minimizer exists because of the continuity and compactness of the uncertainty set. From Assumption~\ref{asm:error} and the construction of $\mathcal{U}(\widehat{P},e)$, for any $(x,a)\in\X\times\A$, we have
\begin{align*}
\| \bar{P}(\cdot|x,a) - P^\star(\cdot|x,a) \|_1 &\le \| \bar{P}(\cdot|x,a) - \widehat{P}(\cdot|x,a) + \widehat{P}(\cdot|x,a) + P^\star(\cdot|x,a) \|_1 \\
&\le \| \bar{P}(\cdot|x,a) - \widehat{P}(\cdot|x,a) \|_1 + \| \widehat{P}(\cdot|x,a) + P^\star(\cdot|x,a) \|_1  \\
&\le 2 e(x,a).
\end{align*}
Applying Lemma~\ref{lem:bound_transitions} with $\pi_1=\pi_2=\pi$, $\M_1=\M^\star$, and $\M_2=\M(\bar{P})$, we obtain
\begin{equation}
\label{eq:rob_proof_lower}
\return(\pi,\M^\star) - \min_{P\in\mathcal{U}(\widehat{P},e)}\return\big(\pi,\M(P)\big) \le \frac{2\gamma \rmax}{1-\gamma} \, p_0\tr (\eye - \gamma P^\star_\pi)^{-1}e_\pi = \frac{2\gamma \rmax}{(1-\gamma)^2} \, \| e_\pi \|_{1,u^\pi_{\M^\star}},    
\end{equation}
where $\pi^\pi_{\M^\star}= (1-\gamma) p_0^\top(\eye - \gamma P^\star_\pi)^{-1}$ is the normalized state occupancy frequency of policy $\pi$ in the true MDP $\M^\star$.

To prove the safety of the returned policy $\pi_{Ra}$, consider the two cases on Line~\ref{ln:algrob_condition} of Algorithm~\ref{alg:robust}. When the condition is satisfied, we have $\return(\pi_B,\M^\star)<\min_{P\in\mathcal{U}(\widehat{P},e)}\return\big(\pi_0,\M(P)\big)\leq\return(\pi_0,\M^\star)$, where the second inequality comes from~\eqref{eq:rob_proof_upper}, and thus, the policy $\pi_R=\pi_0$ is safe. When the condition is violated, then $\pi_R$ is simply $\pi_B$, which is safe by definition. 

To derive a bound on the performance loss of the returned policy $\pi_R$, consider also the two cases on Line~\ref{ln:algrob_condition} of Algorithm~\ref{alg:robust}. When the condition is satisfied, using~\eqref{eq:rob_proof_upper}, we have
\[ \Phi(\pi_R) = \return(\pi^\star_{\M^\star},\M^\star) - \return(\pi_0,\M^\star) \le \return(\pi^\star_{\M^\star},\M^\star) - \min_{P\in\mathcal{U}(\widehat{P},e)}\return\big(\pi_0,\M(P)\big)~, \]
and when the condition is violated, we have 
\[ \Phi(\pi_R) = \return(\pi^\star_{\M^\star},\M^\star) - \return(\pi_B,\M^\star)~. \]

Since the condition is satisfied on Line~\ref{ln:algrob_condition} of Algorithm~\ref{alg:robust} when $\min_{P\in\mathcal{U}(\widehat{P},e)}\return\big(\pi_0,\M(P)\big)>\return(\pi_B,\M^\star)$, we may write 
\[  \Phi(\pi_R) \le \min\left\{\return(\pi^\star_{\M^\star},\M^\star) - \min_{P\in\mathcal{U}(\widehat{P},e)}\return\big(\pi_0,\M(P)\big) ~ \return(\pi^\star_{\M^\star},\M^\star) - \return(\pi_B,\M^\star)\right\}~. \]
Note that we may write the following inequalities for the first term in the minimum
\begin{align*}
\return(\pi^\star_{\M^\star},\M^\star) - \min_{P\in\mathcal{U}(\widehat{P},e)}\return\big(\pi_0,\M(P)\big) 
&\stackrel{\text{(a)}}{\leq} \return(\pi^\star_{\M^\star},\M^\star) - \min_{P\in\mathcal{U}(\widehat{P},e)}\return\big(\pi^\star_{\M^\star},\M(P)\big) \\
&\stackrel{\text{(b)}}{\le} \frac{2\gamma \rmax}{(1-\gamma)^2} \, \| e_{\pi^\star_{\M^\star}} \|_{1,u^\star_{\M^\star}},
\end{align*}
where {\bf (a)} follows from $\pi_0$ being the maximizer in solving the robust MDP, {\bf (b)} is from~\eqref{eq:rob_proof_lower} with $\pi=\pi^\star_{\M^\star}$, and $u^\star_{\M^\star}$ is the normalized state occupancy frequency of the optimal policy $\pi^\star_{\M^\star}$. 
\end{proof}

\subsection{Proof of Proposition \ref{prop:robust-vs-rewadj} } \label{appendix:prop-robust-reward}

\begin{proof}
Let $\bar{P}$ be the minimizer of $\min_{P\in\mathcal{U}(\widehat{P},e)} \return\big(\pi,\M(P)\big)$. Then from Lemma \ref{lem:bound_transitions} for each $\pi$:
\begin{align*}
\rho\big(\pi,\M(\bar{P})\big) \ge \rho(\pi,\widehat{\M}) - \frac{\gamma \rmax}{1-\gamma} p_0\tr (\eye - \gamma \widehat{P}_\pi)^{-1} e_\pi = \rho(\pi,\widetilde{\M}),
\end{align*}
 The last inequality holds because $\widetilde{\M}$ is differs from $\widehat{\M}$ only in its reward function $\hat{r}^\pi=r^\pi-\frac{\gamma\rmax}{1-\gamma}e_\pi$ (see \eqref{eq:rhat}).
\end{proof}


\subsection{The Saddle Point Solution Algorithm}
\label{appendix:saddle_point}

Now we turn to the solution algorithm of \eqref{problem:LB}. Notice that for every fixed $\lambda\geq 0$, the inner optimization problem in \eqref{problem:LB} is a robust MDP problem with $(s,a)-$rectangular uncertainty. We now summarize the standard results of robust value iteration from \cite{nilim2005robust,iyengar2005robust}. 
For any $(x,y)\in\mathcal X\times\mathcal X$, define
the $\lambda-$parametrized robust Bellman's operator $\mathcal T_{\lambda}:\reals^{|\mathcal X\times\mathcal X|}\rightarrow\reals^{|\mathcal X\times\mathcal X|}$ as follows:
\[ 
\mathcal T_{\lambda}[V](x,y)=\max_{a\in\mathcal A}\left\{r^C_{\lambda}(x,y,a)+\gamma\min_{P\in\mathcal U(\widehat P,e)}\sum_{(x^\prime,y^\prime)\in\mathcal X\times\mathcal X}\mathcal P^C(x^\prime,y^\prime|x,y,a)V(x^\prime,y^\prime)\right\}. 
\]
Since the robust Bellman operator is a contraction mapping, its unique fixed point solution equals to the optimal robust Lagrangian function, i.e., 
$\mathcal V_\lambda(x,y)=\mathcal T_\lambda[\mathcal V_\lambda](x,y)$, $\,\forall (x,y)\in\mathcal X\times\mathcal X$ and
\begin{equation}\label{eq:fixed_point_aug}
\sum_{x,x\in\mathcal X}p_0(x)p_0(x)\mathcal V_{\lambda}(x,y)=\min_{P\in\mathcal U(\widehat P,e)}\max_{\pi\in\Pi_S^C} L(\mathcal P^C,\pi,\lambda).
\end{equation}
Thus the robust value function $\mathcal V_{\lambda}(x,y)$ can be calculated by robust value iteration \cite{nilim2005robust},
\[
\mathcal V_{\lambda,0}(x,y)=\mathcal V_{0}(x,y),\,\,\mathcal V_{\lambda,N+1}(x,y)=\mathcal T_{\lambda}[\mathcal V_{\lambda,N}](x,y),\,\,\forall (x,y)\in\mathcal X\times\mathcal X,\,\, N\in\{0,1,2,\ldots,\},
\]
where the $\lambda-$parametrized optimal policy $\pi^{*}_{\lambda}:\mathcal X\times\mathcal X\rightarrow\mathcal A$ and worst case transition probability $ P^\star$ form a \emph{minimax} saddle point of the fixed point solution for $\lambda\geq 0$ \footnote{Note that $\lambda$ is a linear function of $\min_{P\in\mathcal U(\widehat P,e)} L(\pi,\lambda)$ and the worst case minimization is only taken with respect to the constraint component. Then $ P^\star$ is independent of $\lambda$.}.
Furthermore, the following sub-gradient descent algorithm finds the Lagrange multiplier of problem \eqref{problem:LB}:
\begin{enumerate}
\item Find the following Lagrange multiplier update:
\begin{equation*}
\lambda^{(j+1)}=\left(\lambda^{(j)}-\alpha^{(j)} \left(
\return( \pi^{*}_{\lambda^{(j)}},\M(P^\star))-\mathcal \return(\pi_B,\M^\star)
\right)\right)^+
\end{equation*}
 where the step length $\alpha^{(j)}$ is non-summable, square-summable. \footnote{The step-size condition satisfies $\alpha^{(j)}\geq 0$, $\sum_{j=0}^\infty\alpha^{(j)}=\infty$ and $\sum_{j=0}^\infty\left(\alpha^{(j)}\right)^2<\infty$.}
\item Define $f_{\text{min}}^{(j+1)}=\min(f_{\text{min}}^{(j)},f(\lambda^{(j+1)}))$ as the best dual function estimate where $f:\reals_{\geq 0}\rightarrow\reals$ is the robust Lagrangian function at $\pi=\nu_{\lambda}$, i.e.,
\[
f(\lambda)=\max_{\pi\in\Pi_S^C}\min_{P\in\mathcal U(\widehat P,e)} L(\pi,\lambda).
\]
Update the Lagrange multiplier estimate as follows: 
\[
\lambda^{(j+1)}\leftarrow\left\{\begin{array}{ll}
\lambda^{(j+1)}&\text{if  $f_{\text{min}}^{(j+1)}=f(\lambda^{(j+1)})$}\\
\lambda^{(j)}&\text{otherwise}
\end{array}\right..
\]
\end{enumerate}
The following Lemma shows that the solution of the above procedure converges to the solution of $\min_{\lambda\geq 0}f(\lambda)$.
\begin{lemma}\label{lem:tech2}
Let $\lambda^\star$  be the minimizer of $\min_{\lambda\geq 0}f(\lambda)$. The solution of the projected sub-gradient descent algorithm converges to $\lambda^\star$, i..e, $\lim_{j\rightarrow\infty}\lambda^{(j)}=\lambda^\star$.
\end{lemma}
\begin{proof}
Since $ L(\pi,\lambda)$ is a linear  function of $\lambda$ for any given $\mathcal P^C$ and $\pi$, we have that $f(\lambda)=\max_{\pi\in\Pi_S^C}\min_{P\in\mathcal U(\widehat P,e)}  L(\pi,\lambda)$ is a convex function in $\lambda$.
By the envelop theorem of mathematical economics \cite{milgrom2002envelope}, for $\lambda \geq 0$ one can write
\[
\frac{d f(\lambda)}{d\lambda}=\min_{P\in\mathcal U(\widehat P,e)}\return(\pi,\M(P))-\mathcal \return(\pi_B,\M^\star).
\]

Now, we show that the proposed sub-gradient descent algorithm converges to the optimizer of  $\min_{\lambda\geq 0}f(\lambda)$. First, one can write the $\lambda$ iterate as $\lambda^{(j+1)}=\left(\bar\lambda^{(j+1)}\right)^+$ where
\[
\bar\lambda^{(j+1)}:=\lambda^{(j)}-\alpha^{(j)}\frac{d f(\lambda)}{d\lambda}\bigg\vert_{\lambda=\lambda^{(j)}}.
\]
Since the projection operator is non-expansive, one obtains $(\lambda^{(j+1)}-\lambda^\star)^2\leq(\bar\lambda^{(j+1)}-\lambda^\star)^2$. 
Furthermore, the following expression holds:
\[
\begin{split}
(\lambda^{(j+1)}-\lambda^\star)^2\leq &(\bar\lambda^{(j+1)}-\lambda^\star)^2\\
=&\left(\lambda^{(j)}-\alpha^{(j)}\frac{d f(\lambda)}{d\lambda}\bigg\vert_{\lambda=\lambda^{(j)}}-\lambda^\star\right)^2\\
=&(\lambda^{(j)}-\lambda^\star)^2-2\alpha^{(j)}(\lambda^{(j)}-\lambda^\star)\frac{d f(\lambda)}{d\lambda}\bigg\vert_{\lambda=\lambda^{(j)}}+\left(\alpha^{(j)}\right)^2\left(\frac{d f(\lambda)}{d\lambda}\bigg\vert_{\lambda=\lambda^{(j)}}\right)^2\\
\leq &(\lambda^{(j)}-\lambda^\star)^2-2\alpha^{(j)}(f(\lambda^{(j)})-f(\lambda^\star))+\left(\alpha^{(j)}\right)^2\left(\frac{d f(\lambda)}{d\lambda}\bigg\vert_{\lambda=\lambda^{(j)}}\right)^2,
\end{split}
\]
The inequality is due to the fact that for any $\lambda\geq 0$, convexity of $f(\lambda)$ implies that
\[
f(\lambda)-f(\lambda^\star)\geq (\lambda-\lambda^\star)\frac{d f(\lambda)}{d\lambda}.
\]
This further implies
\[
(\lambda^{(j+1)}-\lambda^\star)^2\leq (\lambda^{(0)}-\lambda^\star)^2-\sum_{q=0}^j2\alpha^{(q)}(f(\lambda^{(q)})-f(\lambda^\star))+\left(\alpha^{(q)}\right)^2\left(\frac{d f(\lambda)}{d\lambda}\bigg\vert_{\lambda=\lambda^{(q)}}\right)^2.
\]
Since $(\lambda^{(j+1)}-\lambda^\star)^2$ is a positive quantity and $(\lambda^{(0)}-\lambda^\star)^2$ is bounded, this further implies
\[
2\sum_{q=0}^j\alpha^{(q)}(f(\lambda^{(q)})-f(\lambda^\star))\leq (\bar\lambda^{(0)}-\lambda^\star)^2+\sum_{q=0}^j\left(\alpha^{(q)}\right)^2\left(\frac{d f(\lambda)}{d\lambda}\bigg\vert_{\lambda=\lambda^{(q)}}\right)^2.
\]
By defining $f_{\text{min}}^{(j)}=\min_{q\in\{0,\ldots,j\}}f(\lambda^{(q)})$, the above expression implies
\[
f_{\text{min}}^{(j)}-f(\lambda^\star)\leq \frac{1}{\sum_{q=0}^j\alpha^{(q)}}\left((\lambda^{(0)}-\lambda^\star)^2+\sum_{q=0}^j\left(\alpha^{(q)}\right)^2\left(\frac{d f(\lambda)}{d\lambda}\bigg\vert_{\lambda=\lambda^{(q)}}\right)^2\right).
\]
The step-size rule of $\alpha^{(q)}$ ensures that the numerator is bounded and the denominator goes to infinity as $j\rightarrow\infty$.
This implies that for any $\epsilon>0$, there exists a constant $N(\epsilon)$ such that for any $j>N(\epsilon)$, $f(\lambda^\star) \leq f_{\text{min}}^{(j)}\leq f(\lambda^\star)+\epsilon$.
In other words, the sequence $\lambda^{(j)}$ converges to the global minimum $\lambda^\star$ of $f(\lambda)$. 
\end{proof}
Combining all previous arguments, the saddle point solution of \eqref{problem:LB} is given by $(\pi_\lambda^\star,\lambda^\star)$.


\subsection{Proof of Theorem \ref{thm:perf_robust_aug}}
\label{appendix:robust_aug_perf}

\begin{proof}
Given Assumption~\ref{asm:error}, if the solution $\min_{\lambda\geq 0}\max_{\pi\in\Pi_S^A}\min_{P\in\mathcal U(\widehat{P},e)}L_P(\pi,\lambda)$ is lower-bounded, then weak duality implies that the primal Lagrangian $\max_{\pi\in\Pi_S^A}\min_{\lambda\geq 0} L(\pi,\lambda)$ is also lower-bounded, which further implies that $\pi_{AR}$ is a {\em safe} policy. Otherwise, Algorithm~\ref{alg:robust_aug} returns the baseline policy $\pi_B$. This concludes that the policy $\pi_{AR}$ returned by Algorithm~\ref{alg:robust_aug} is safe.

For the performance loss bound, without loss of generality it is analyzed based on MDP $\M^A_{\lambda_1,\lambda_2}(P)$ with $\lambda_1 = 1$ and $\lambda_2 =0$, where $\M^A_{1,0}(P)=\M(P)$ and $\M^A_{1,0}(P^\star)=\M^\star$. The proof follows with arguments identical to those in the proof of Theorem~\ref{thm:perf_robust} and is omitted for the sake of brevity. 
\end{proof}


\subsection{Proof of Proposition~\ref{prop:robust-vs-augmented_robust}}
\label{appendix:robust_aug_feasibility}

\begin{proof}
Suppose that Algorithm~\ref{alg:robust} returns a safe policy other than $\pi_B$, i.e.,~$\pi_R\neq\pi_B$. Since the class of policies $\Pi_S^A$ is dominating for the optimization problem~\eqref{eq:1_sim_real}, we may write~\eqref{eq:1_sim_real} as 
\begin{equation}\label{problem:ori}
\max_{\pi\in\Pi_S^{AR}} \return( \pi,\widehat{\M}),
\end{equation}
where the feasible policy set $\Pi_S^{AR}$ is defined as
\begin{equation*}
\Pi_S^{AR}:=\Big\{\pi\in\Pi_S^A:\min_{P\in\mathcal U(\widehat P,e)} \return\big(\pi,\M(P)\big)\geq \return(\pi_B,\M^\star)\Big\}.
\end{equation*}
This leads to the primal Lagrangian formulation $\max_{\pi\in\Pi_S^A} \min_{\lambda\geq 0}\min_{P\in\mathcal U(\widehat P,e)}L_P(\pi,\lambda)$. Since $\pi_R\neq\pi_B$ is a safe policy and the feasible set $\Pi_S^{AR}$ is non-empty, the solution of the primal Lagrangian is equal to the solution of~\eqref{problem:ori}. Furthermore by weak duality, we have 
\begin{equation*}
\max_{\pi\in\Pi_S^A} \min_{\lambda\geq 0}\min_{P\in\mathcal U(\widehat P,e)}L_P(\pi,\lambda) \leq\underbrace{\min_{\lambda\geq 0}\max_{\pi\in\Pi_S^A}\min_{P\in\mathcal U(\widehat P,e)}L_P(\pi,\lambda)}_{\text{solution of Algorithm~\ref{alg:robust_aug}}}.
\end{equation*}

On the other hand, consider the following optimization problem:
\begin{equation}
\label{robust_pro_2}
\max_{\pi\in\Pi_S^{AR}}\min_{\P\in\mathcal U(\widehat P,e)} \return\big(\pi,\M(P)\big).
\end{equation} 
When $\pi_R\neq\pi_B$ is a safe policy, the feasible set $\Pi_S^{AR}$ is non-empty, and we may write the Lagrangian function of~\eqref{robust_pro_2} as
\begin{equation*}
\mathcal L(\pi,\lambda)=\min_{\P\in\mathcal U(\widehat P,e)} \return\big(\pi,\M(P)\big)+\lambda\left(\min_{\P\in\mathcal U(\widehat P,e)} \return\big(\pi,\M(P)\big)-\return(\pi_B,\M^\star)\right)
\end{equation*}
Since the feasible set is non-empty, we may drop the constraint in~\eqref{robust_pro_2} and this problem becomes equivalent to
\begin{equation*}
\max_{\pi\in\Pi_S^A}\min_{\P\in\mathcal U(\widehat P,e)} \return\big(\pi,\M(P)\big),
\end{equation*}
where the dominating class of policies of this problem is stationary Markovian $\Pi_S$. Thus the primal Lagrangian duality formulation implies that
\begin{equation*}
\max_{\pi\in\Pi_S^A} \min_{\lambda\geq 0} \mathcal L(\pi,\lambda)= \max_{\pi\in\Pi_S^A}\min_{P\in\mathcal U(\widehat{P},e)}\return\big(\pi,\M(P)\big)=\max_{\pi\in\Pi_S}\min_{P\in\mathcal U(\widehat{P},e)}\return\big(\pi,\M(P)\big),
\end{equation*}
i.e., $\max_{\pi\in\Pi_S^A} \min_{\lambda\geq 0} \mathcal L(\pi,\lambda)$ equals to the solution of Algorithm~\ref{alg:robust}. Since the objective function of~\eqref{robust_pro_2} is lower than that of~\eqref{problem:ori} and both problems share the same feasible set, it is obvious that 
\begin{equation*}
\max_{\pi\in\Pi_S^A} \min_{\lambda\geq 0}\mathcal L(\pi,\lambda)\leq \max_{\pi\in\Pi_S^A} \min_{\lambda\geq 0}\min_{P\in\mathcal U(\widehat P,e)}L_P(\pi,\lambda).
\end{equation*}
Combining these arguments, we have just showed that if $\pi_R$ is a safe policy, then
\begin{equation*}
\return(\pi_B,\M^\star)\leq \min_{P\in\mathcal U(\widehat{P},e)}\return\big(\pi_R,\M(P)\big)\leq \min_{\lambda\geq 0}\max_{\pi\in\Pi_S^A}\min_{P\in\mathcal U(\widehat P,e)}L_P(\pi,\lambda)\leq\min_{P\in\mathcal U(\widehat{P},e)} L_P(\pi_{AR},\lambda^\star),
\end{equation*}
where $(\pi_{AR},\lambda^\star)$ is the \emph{maximin} saddle-point solution of $\min_{P\in\mathcal U(\widehat P,e)}L_P(\pi,\lambda)$.
\end{proof}


\subsection{Proof of Theorem \ref{thm:perf_robust_interleave}}
\label{appendix:robust_aug_interleave}

\begin{proof}
Consider an arbitrary $\widetilde{P} \in \mathcal{U}(\widehat{P},e)$. Then from Assumption \ref{asm:error} and the construction of $\mathcal{U}(\widehat{P},e)$, we have:
\begin{align*}
\| \widetilde{P}(\cdot|x,a) - P\opt(\cdot|x,a) \|_1 &\le \| \widetilde{P}(\cdot|x,a) - \widehat{P}(\cdot|x,a) + \widehat{P}(\cdot|x,a) + P\opt(\cdot|x,a) \|_1 \\
&\le \| \widetilde{P}(\cdot|x,a) - \widehat{P}(\cdot|x,a) \|_1 + \| \widehat{P}(\cdot|x,a) + P\opt(\cdot|x,a) \|_1  \\
&\le 2 e(x,a) ~.
\end{align*}
Then, using Lemma \ref{lem:bound_transitions} with the above difference between $\widetilde{P}$ and $P\opt$,  we get for any policy $\pi$ that:
\begin{equation} \label{eq:robi_proof_upper}
\max_{P\in\mathcal{U}( \widehat{P},e )} \Bigl|  \rho(\pi, P) - \rho(\pi, P\opt) \Bigr|  \le \frac{2\gamma \rmax}{1-\gamma} \, p_0\tr (\eye - P_\pi\opt ) \, e_\pi = \frac{2\gamma \rmax}{(1-\gamma)^2} \, \| e_\pi \|_{1,u}  ~,    
\end{equation}
where $u_\pi$ is the normalized state occupancy frequency for policy $\pi$ defined as:
\[ u = (1-\gamma) (\eye - \gamma P_\pi\tr)^{-1} p_0. \]

To prove the safety of $\polrobi$, note that the objective in \eqref{eq:objective_robust_interleave} is always non-negative since $\pi_B$ is feasible. Then we get the safety condition by simple algebraic manipulation as follows:
\begin{align*}
\min_{P\in\mathcal{U}( \widehat{P},e )} \Big( \return(\polrobi, P) - \return(\pi_B, P) \Big) &\ge 0 \\
\return(\polrobi, P\opt) &\ge \return(\pi_B, P\opt)
\end{align*}
The safety of the policy $\polrobi$ also implies that its performance loss is bounded by the performance loss of the base policy:
\begin{equation} \label{eq:performance_loss_robi}
\return(\polrobi, P\opt) - \return(\polrobi, P\opt)  \le \return(\polrobi, P\opt) - \return(\pi_B, P\opt)
\end{equation} 

Now  we ready to show a bound on the performance loss of $\polrobi$ by lower bounding $\return(\polrobi,P\opt)$ as follows:
\[ \return(\polrobi,P\opt) = \return(\polrobi,P\opt)  - \return(\pi_B,P\opt) + \return(\pi_B,P\opt) \ge \min_{P\in\mathcal{U}(\widehat{P},e)} \Big( \return(\polrobi,P)  - \return(\pi_B,P) \Big) + \return(\pi_B,P\opt) ~. \] 
From the optimality of $\polrobi$, we further get:
\begin{align*}  
\min_{P\in\mathcal{U}(\widehat{P},e)} \Big( \return(\polrobi,P)  - \return(\pi_B,P) \Big) &\ge \min_{P\in\mathcal{U}(\widehat{P},e)} \Big( \return(\pol\opt,P) - \return(\pi_B,P) \Big) \\
&\ge \min_{P\in\mathcal{U}(\widehat{P},e)} \return(\pol\opt,P) - \max_{P\in\mathcal{U}(\widehat{P},e)} \return(\pi_B,P) ~.
\end{align*}
Putting the above together and some simple algebraic manipulation subtracting and adding $\return(\pol\opt,P\opt)$, we get:
\[ \return(\pol\opt,P\opt) - \return(\polrobi,P\opt) \le \max_{P\in\mathcal{U}(\widehat{P},e)} \Big( \return(\pol\opt,P\opt) - \return(\pol\opt, P) \Big) + \max_{P\in\mathcal{U}(\widehat{P},e)} \Big( \return(\pi_B, P) - \return(\pi_B,P\opt)  \Big). \]
The bound in the theorem then follows by bounding the maximization terms above using \eqref{eq:robi_proof_upper} and combining the above inequality with \eqref{eq:performance_loss_robi}.

Figure \ref{fig:example_tight_bound} depicts an example demonstrating the tightness of the bound. The initial state is $s_0$, actions are $a_1,a_2$, and the transitions are deterministic. $\xi\opt$ denotes the true transitions, $\xi_1$ denotes the worst case in $\mathcal{U}(\widehat{P},e)$, and the leafs shows the returns of the remainder of the MDP and are assumed to be known with certainty. The value for $\epsilon$ is given by \eqref{eq:robi_proof_upper} . That is $\widehat{\xi}$ represents the sampled transition probability, it would be halfway between $\xi\opt$ and $\xi_1$.
\end{proof}

\begin{figure}
\centering
\begin{tikzpicture}[->,>=stealth',shorten >=1pt,auto,node distance=2cm,semithick]
	\tikzstyle{action}= [circle,fill=black,text=white]

	\node[state] 		 (s1)                    														{$s_0$};
	\node[action]      (a1)         [below left of=s1]     	{$a_1$};
	\node[action]      (a2)         [below right of=s1]    {$a_2$};

	\node (o11) [below left of=a1,xshift=5mm] {$1$};
	\node (o12) [below right of=a1,xshift=-5mm] {$1+\epsilon$};
	\node (o21) [below left of=a2,xshift=5mm] {$1+2\epsilon$};
	\node (o22) [below right of=a2,xshift=-5mm] {$1+\epsilon$};

	\path (s1) edge node [left]{$\pi_{\text{B}},\pi_{\text{I}}$} (a1) ;
	\path (a1) edge node [left]{$\xi^\star$} (o11);
	\path (a1) edge node [right]{$\xi_1$} (o12);

	\path (s1) edge node [right]{$\pi^\star$} (a2) ;
	\path (a2) edge node [left]{$\xi_1$} (o21);
	\path (a2) edge node [right]{$\xi^\star$} (o22);

\end{tikzpicture}
\caption{Example showing tightness of bound in Theorem \ref{thm:perf_robust_interleave}.} \label{fig:example_tight_bound}
\end{figure}
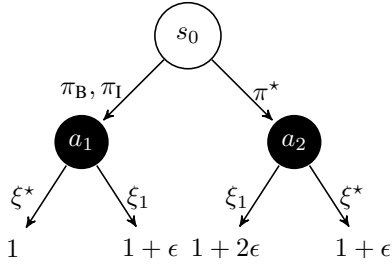

\section{Alternative Bounds}

\subsection{An Alternative Bound on the Performance Loss of $\pi_{Ra}$}
\label{appendix:ThmPerfRewAdj2}

\begin{corollary}\label{coro:perf_reward_adjusted}
Given Assumption~\ref{asm:error}, the performance loss $\Phi(\pi_{Ra})$ also satisfies 
\[
\Phi(\pi_{Ra}) \le \min \left\{ \frac{\text{BR}(\pi_{Ra})}{1-\gamma}+\max_{\pi\in\Pi_S}\frac{2\gamma\rmax}{(1-\gamma)^2}  \| e_{\pi_{\widehat\M}} \|_{1,\pi^\pi_{\widehat\M}},  \Phi(\pi_B) \right\},
\]
where $\text{BR}(\pi_{Ra})=\max_{x\in\mathcal X}\left|\widehat{T}[\widehat{V}^{\pi_{Ra}}](x)-\widehat{V}^{\pi_{Ra}}(x)\right|$ is the Bellman residual w.r.t. Bellman operator $\widehat{T}[V](x)=\max_{a\in\mathcal A}\left\{\widehat r(x,a)+\gamma\sum_{x^\prime\in\mathcal X}\widehat{P}(x^\prime|x,a)V(x^\prime)\right\}$, value function $\widehat{V}^{\pi_{Ra}}(x)=\rho(\pi_{Ra},\widehat\M)$ at $x_0=x$ and $\pi^\pi_{\widehat\M}$ is the normalized state occupancy frequency of the policy $\pi$.
\end{corollary}

\begin{proof}
Similar to the proof of  Theorem \ref{coro:perf_reward_adjusted}, we notice that
\begin{equation}\label{bdd:eval_sim}
\return(\pi^*_{\M^\star},\M^\star) - \return(\pi_0,\widetilde{\M}) =\underbrace{\return(\pi^*_{\M^\star},\M^\star)- \max_{\pi\in\Pi_S}\return(\pi,\widetilde{\M})}_{(a)} +\underbrace{\max_{\pi\in\Pi_S}\return(\pi,\widetilde{\M}) - \return(\pi_0,\widetilde{\M}) }_{(b)}
\end{equation}

First we prove an upper bound for (a). Recall the Bellman operator $\widehat T$ for value function $V:\reals^{\mathcal X}\rightarrow\reals^{\mathcal X}$ as $\widehat{T}[V](y)=\max_{a\in\mathcal A}\left\{\widehat r(y,a)+\gamma\sum_{y^\prime\in\mathcal X}\widehat{P}(y^\prime|y,a)V(y^\prime)\right\}$. Also define the value function $\widehat{V}^{\pi_0}(x)$ as $\return(\pi_0,\widetilde{\M})$ and the optimal value function $\widehat{V}(x)$ as $\max_{\pi\in\Pi_S}\return(\pi,\widetilde{\M})$, when the initial state is $x_0=x$. 
By applying the contraction mapping property on $\widehat{T} [\widehat{V}^{\pi_0}](y)-\widehat{V}^{\pi_0}(y)$ for any $y\in\mathcal X$ and combing with the definition of Bellman residual, one obtains
\[
\widehat{T}^2 [\widehat{V}^{\pi_0}](y)-\widehat{T}[\widehat{V}^{\pi_0}](y)\leq \gamma\text{BR}(\pi_0).
\]
By an induction argument, the above expression becomes 
\[
\widehat{T}^N[\widehat{V}^{\pi_0}](y)-\widehat{T}^{N-1}[\widehat{V}^{\pi_0}](y)\leq \gamma^{N-1}\text{BR}(\pi_0),
\]
for which by a telescoping sum, it further implies
\[
\widehat{T}^N [\widehat{V}^{\pi_0}](y)-\widehat{V}^{\pi_0}(y)= \sum_{k=1}^{N}\widehat{T}^k [\widehat{V}^{\pi_0}](y)-\widehat{T}^{k-1}[\widehat{V}^{\pi_0}](y)
\leq  \sum_{k=1}^N\gamma^{k-1}\text{BR}(\pi_0).
\]
By letting $N\rightarrow\infty$ and noticing that $\lim_{N\rightarrow\infty}\widehat{T}^N [\widehat{V}^{\pi_0}](y)=\widehat{V}(y)$, one finally obtains
\begin{equation}\label{eq:BR}
(b) \le\left|\max_{\pi\in\Pi_S}\return(\pi,\widetilde{\M}) - \return(\pi_0,\widetilde{\M})\right|=\left|\sum_{y\in\mathcal X} P_0(y)\left(\widehat{V}(y)-\widehat{V}^{\pi_0}(y)\right)\right|\leq  \lim_{N\rightarrow\infty}\sum_{k=1}^N\gamma^{k-1}\text{BR}(\pi_0)= \frac{\text{BR}(\pi_0)}{1-\gamma}.
\end{equation}
For an upper bound in (a), by interchanging $\M$ and $\widetilde\M$ in the derivation of \ref{eq:exp_proof_lower} and taking maximum on both sides, we also have,
\[
(a)=\max_{\pi\in\Pi_S}\return(\pi,{\M}^\star)-\max_{\pi\in\Pi_S}\return(\pi,\widetilde{\M})\le\max_{\pi\in\Pi_S}\frac{2\gamma\rmax}{1-\gamma} p_0\tr (\eye - \disc \widehat P_{\pi})^{-1} e_{\pi}\leq\max_{\pi\in\Pi_S}\frac{2\gamma\rmax}{(1-\gamma)^2}  \| e_{\pi_{\widehat\M}} \|_{1,\pi^\pi_{\widehat\M}} .
\]
Then the proof is completed by  combining both parts of the above arguments.
\end{proof}

\subsection{An Alternative Bound on the Performance Loss of $\pi_R$}
\label{appendix:ThmPerfRobust2}

\begin{corollary}\label{coro:perf_robust}
Given Assumption~\ref{asm:error}, the performance loss $\Phi(\pi_{R})$ also satisfies 
\[
\Phi(\pi_{R}) \le \min \left\{ \frac{\text{BR}(\pi_{R})}{1-\gamma}+\max_{\pi\in\Pi_S}\min_{P\in\mathcal{U}(\widehat{P},e)}\frac{2\gamma\rmax}{(1-\gamma)^2} \| e_{\pi_{\M(P)}} \|_{1,\pi_{\M(P)}},  \Phi(\pi_B) \right\},
\]
where $\text{BR}(\pi_{R})=\max_{x\in\mathcal X}\left|\mathcal{T}[{V}^{\pi_{R}}](x)-{V}^{\pi_{R}}(x)\right|$ is the Bellman residual w.r.t. Bellman operator $\mathcal{T}[V](x)=\max_{a\in\mathcal A}\left\{ r(x,a)+\gamma\min_{P\in\mathcal U(\widehat P,e)}\sum_{x^\prime\in\mathcal X}{P}(x^\prime|x,a)V(x^\prime)\right\}$, value function ${V}^{\pi_{R}}(x)=\min_{P\in\mathcal U(\widehat P,e)}\rho(\pi_{R},\M(P))$ at $x_0=x$ and $u^\pi_{\M(P)}$ is the normalized state occupancy frequency of the policy $\pi$.
\end{corollary}
\begin{proof}
For the proof of the alternative performance bound, by following the same analysis as in the derivation of \eqref{bdd:eval_sim}, replacing the bellman operator $\widehat T$ with the robust Bellman operator:
\[
\mathcal{T}[V](y)=\max_{a\in\mathcal A}\left\{ r(y,a)+\gamma\min_{P\in\mathcal U(\widehat P,e)}\sum_{y^\prime\in\mathcal X}{P}(y^\prime|y,a)V(y^\prime)\right\},
\]
and defining the robust value function ${V}^{\pi_0}(x)$ as $\min_{P\in\mathcal{U}(\widehat{P},e)}\return\big(\pi_0,{\M}(P)\big)$ when the initial state is $x_0=x$,
we can easily show that
\[
\max_{\pi\in\Pi_S}\min_{P\in\mathcal{U}(\widehat{P},e)}\return(\pi,{\M}(P)) - \min_{P\in\mathcal{U}(\widehat{P},e)}\return(\pi_0,{\M}(P))\leq\frac{\text{BR}(\pi_0)}{1-\gamma},
\]
and
\[
\begin{split}
\max_{\pi\in\Pi_S}\return(\pi,{\M}^\star)-\max_{\pi\in\Pi_S}\min_{P\in\mathcal{U}(\widehat{P},e)}\return\big(\pi,\M(P)\big) \le&\max_{\pi\in\Pi_S}\min_{P\in\mathcal{U}(\widehat{P},e)}\frac{2\gamma\rmax}{1-\gamma} p_0\tr (\eye - \disc  P_{\pi})^{-1} e_{\pi}\\
\leq& \max_{\pi\in\Pi_S}\min_{P\in\mathcal{U}(\widehat{P},e)}\frac{2\gamma\rmax}{(1-\gamma)^2}  \| e_{\pi_{\M(P)}} \|_{1,u^\pi_{\M(P)}}.
\end{split}
\]
The proof is then completed by combining both of the above results.
\end{proof}

\subsection{An Alternative Bound on the Performance Loss of $\pi_{AR}$}
\label{appendix:robust_aug_perf2}


\begin{corollary}
\label{coro:perf_robust_aug}
Given Assumption~\ref{asm:error}, the performance loss $\Phi(\pi_{AR})$ also satisfies 
\[
\Phi(\pi_{AR}) \le \min \left\{ \frac{\text{BR}(\pi_{AR})}{1-\gamma}+\max_{\pi\in\Pi_S}\min_{P\in\mathcal{U}(\widehat{P},e)}\frac{2\gamma\rmax}{(1-\gamma)^2} \| e_{\pi_{\M(P)}} \|_{1,\pi_{\M(P)}},  \Phi(\pi_B) \right\}.
\]
where $u^\pi_{\M(P)}$ is the normalized state occupancy frequency of the policy $\pi$ and the Bellman residual for robust MDPs is $\text{BR}(\pi_{AR})=\text{BR}_{1,0}(\pi_{AR})$, for which the generic case $(\lambda_1,\lambda_2)\geq 0$, the Bellman residual w.r.t. Bellman operator $\mathcal{T}_{\lambda_1\lambda_2}[V](x,y)=\max_{a\in\mathcal A}\left\{ r^A_{\lambda_1,\lambda_2}(x,y,a)+\gamma\min_{P\in\mathcal U(\widehat P,e)}\sum_{x^\prime,y^\prime\in\mathcal X}{P}^A(x^\prime,y^\prime|x,y,a)V(x^\prime,y^\prime)\right\}$ is given by $\text{BR}_{\lambda_1,\lambda_2}(\pi_{RS})=\max_{x,y\in\mathcal X}\left|\mathcal{T}_{\lambda_1\lambda_2}[{V}_{\lambda_1\lambda_2}^{\pi_{AR}}](x,y)-{V}_{\lambda_1\lambda_2}^{\pi_{AR}}(x,y)\right|$ where ${V}^{\pi_{AR}}(x,y)=\min_{P\in\mathcal U(\widehat P,e)}\rho\big(\pi_{AR},\M^A_{\lambda_1\lambda_2}(P)\big)$ is the value function at $x_0=x$.
\end{corollary}

\begin{proof}
Since to the proof of Theorem~\ref{thm:perf_robust_aug}, the proof of this corollary follows identical arguments from Corollary~\ref{coro:perf_robust} and is omitted for the sake of brevity.
\end{proof}

\end{document}

%% file: defs.tex
\newcommand{\BEAS}{\begin{eqnarray*}}
\newcommand{\EEAS}{\end{eqnarray*}}
\newcommand{\BEQ}{\begin{equation}}
\newcommand{\EEQ}{\end{equation}}
\newcommand{\BIT}{\begin{itemize}}
\newcommand{\EIT}{\end{itemize}}

\usepackage{color}
\usepackage{amsmath}
\usepackage{amssymb}
\usepackage{graphicx}
\usepackage{comment,xspace}
\usepackage{fancybox}

\newcommand{\real}{{\mathbb{R}}}
\newcommand{\reals}{\real}

%% file: RL_var.tex
\def\A{\mathcal{A}}

\def\P{P}

\def\X{\mathcal{X}}
\def\U{\mathcal{U}}
\def\M{\mathcal{M}}
